\documentclass[reqno]{amsart}
\usepackage{amsmath, amsthm, amssymb, color}
\usepackage{graphicx}
\usepackage[mathscr]{euscript}
\usepackage{latexsym}
\usepackage{hyperref}
\hypersetup{pdfstartview={XYZ null null 1.00}} 
\usepackage{xcolor}
\usepackage{amsmath,amssymb,subeqnarray,amsfonts,graphics,epsfig,comment,color}
\usepackage{subfigure,txfonts}
\usepackage{latexsym}
 \usepackage{mathptmx}      % use Times fonts if available on your TeX system
\usepackage{color}
\usepackage[numbers,sort&compress]{natbib}
\numberwithin{equation}{section}
\newtheorem{Theorem}{Theorem}[section]
\newtheorem{Lemma}[Theorem]{Lemma}

\theoremstyle{definition}

\newtheorem{remark}[Theorem]{Remark}
\newtheorem{Proposition}[Theorem]{Proposition}

\providecommand{\norm}[1]{\left\Vert#1\right\Vert}

\newcounter{RomanNumber}

\def\be{\begin{equation}}
\def\en{\end{equation}}
\def\bs{\begin{split}}
\def\es{\end{split}}

\allowdisplaybreaks
\title[Global existence and decay rates for a generic non--conservative compressible two--fluid model]
{Global existence and decay rates for a generic compressible
two--fluid model}
\author{Yin Li}
\address{Yin Li \newline Faculty of Education, Shaoguan University,
512005, Shaoguan, P. R. China.} \email{liyin2009521@163.com}
\author{Huaqiao Wang}
\address{Huaqiao Wang \newline College of Mathematics and Statistics, Chongqing University,
Chongqing 401331, China.} \email{wanghuaqiao@cqu.edu.cn}
\author{Guochun Wu}
\address{Guochun Wu \newline Fujian Province University Key Laboratory of Computational Science, School of Mathematical Sciences, Huaqiao University, Quanzhou 362021, P.R. China.}
\email{guochunwu@126.com}
\author{Yinghui Zhang*}
\address{Yinghui Zhang \newline Center for Applied Mathematics of Guangxi, Guangxi Normal University, Guilin, Guangxi 541004, P.R.
China} \email{yinghuizhang@mailbox.gxnu.edu.cn}

\thanks{* Corresponding author: yinghuizhang@mailbox.gxnu.edu.cn}
\bigbreak

\date{\today}
\usepackage{hyperref}

\begin{document}
\begin{abstract}
We investigate global existence and optimal decay rates of a generic
non-conservative compressible two--fluid model with general constant
viscosities and capillary coefficients. Bresch, et al. in the
seminal work (Arch Rational Mech Anal 196:599--629, 2010) considered
the compressible two--fluid model with a special type  of
density--dependent viscosities ($\mu^\pm(\rho^\pm)=\mu^\pm
\rho^\pm,~~\lambda^\pm=0$). However, as indicated by themselves,
their methods cannot deal with the case of constant viscosity
coefficients. Besides, Cui, et al. (SIAM J Math Anal 48:470--512,
2016) studied the same model with a more special type of
density-dependent viscosity
$(\mu^\pm(\rho^\pm)=\nu\rho^\pm,~~\lambda^\pm=0$) and equal
capillary coefficients ($\sigma^+=\sigma^-=\sigma$). Since their
analysis relies heavily this special choice for viscosities and
capillary coefficients, the case of general constant viscosities and
capillary coefficients cannot be handled in their settings. The main
novelty of this work is three--fold: First, for any integer
$\ell\geq3$, we show that the densities and velocities converge to
their corresponding equilibrium states at the $L^2$ rate
$(1+t)^{-\frac{3}{4}}$, and the $k$($\in [1, \ell]$)--order spatial
derivatives of them converge to zero at the $L^2$ rate
$(1+t)^{-\frac{3}{4}-\frac{k}{2}}$, which are the same as ones of
the compressible Navier--Stokes system, Navier--Stokes--Korteweg
system and heat equation. Second, the linear combination of the
fraction densities ($\beta^+\alpha^+\rho^++\beta^-\alpha^-\rho^-$)
converges to its corresponding equilibrium state at the $L^2$ rate
$(1+t)^{-\frac{3}{4}}$, and its $k$($\in [1, \ell]$)--order spatial
derivative converges to zero at the $L^2$ rate
$(1+t)^{-\frac{3}{4}-\frac{k}{2}}$, but the fraction densities
($\alpha^\pm\rho^\pm$) themselves converge to their corresponding
equilibrium states at the $L^2$ rate $(1+t)^{-\frac{1}{4}}$, and the
$k$($\in [1, \ell]$)--order spatial derivatives of them converge to
zero at the $L^2$ rate $(1+t)^{-\frac{1}{4}-\frac{k}{2}}$, which are
slower than ones of their linear combination
($\beta^+\alpha^+\rho^++\beta^-\alpha^-\rho^-$) and the densities.
We think that this phenomenon should owe to the special structure of
the system. Finally, for well--chosen initial data, we also prove
the lower bounds on the decay rates, which are the same as those of
the upper decay rates. Therefore, these decay rates are optimal for
the compressible two--fluid model.

\end{abstract}

\maketitle

%%%%%%%%%%%%%%%%%%%%%%%%%%%%%%%%%%%%%%%%%%%%%%%
%%%%%%%%%%%%%%%%%%%%%%%%%%%%%%%%%%%%%%%%%%%%%%%
\section{\leftline {\bf{Introduction.}}}
\setcounter{equation}{0}
%%%%%%%%%%%%%%%%%%%%%%%%%%%%%%%%%%%%%%%%%%%%%%%
\subsection{Background and motivation}
As is well--known, multi--fluid flows are very common in nature.
Such a terminology includes the flows of non--miscible fluids such
as air and water; gas, oil and water. For the flows of miscible
fluids, they usually form a ``new" single fluid possessing its own
rheological properties. One interesting example is the stable
emulsion between oil and water which is a non--Newtonian fluid, but
oil and water are Newtonian ones.\par
 One of the classic examples of multi--fluid flows is
small amplitude waves propagating at the interface between air and
water, which is called a separated flow. In view of modeling, each
fluid obeys its own equation and couples with each other through the
free surface in this case. Here, the motion of the fluid is governed
by the pair of compressible Euler equations with free surface:
\begin{align}
\partial_{t} \rho_{i}+\nabla \cdot\left(\rho_{i} v_{i}\right) &=0, \quad i=1,2,\label{1.1} \\
\partial_{t}\left(\rho_{i} v_{i}\right)+\nabla \cdot\left(\rho_{i} v_{i} \otimes v_{i}\right)+\nabla p_i &=-g\rho_{i}
e_3\pm F_D.\label{1.2}
\end{align}
In above equations, $\rho_1$ and $v_1$ represent the density and
velocity of the upper fluid (air), and  $\rho_2$ and $v_2$ denote
the density and velocity of the lower fluid (water). $p_{i}$ denotes
the pressure. $-g\rho_{i} e_3$ is the gravitational force with the
constant $g>0$ the acceleration of gravity and $e_3$ the vertical
unit vector, and $F_D$ is the drag force. As mentioned before, the
two fluids (air and water) are separated by the unknown free surface
$z=\eta(x, y, t)$, which is advected with the fluids according to
the kinematic relation:
\begin{equation}\partial_t\eta=u_{1,z}-u_{1,x}\partial_x \eta-u_{1, y}\partial_y \eta\label{1.3}\end{equation}
on two sides of the surface $z=\eta$ and the pressure is continuous
across this surface.\par When the wave's amplitude becomes large
enough, wave breaking may happen. Then, in the region around the
interface between air and water, small droplets of liquid appear in
the gas, and bubbles of gas also appear in the liquid. These
inclusions might be quite small. Due to the appearances of collapse
and fragmentation, the topologies of the free surface become quite
complicated and a wide range of length scales are involved.
Therefore, we encounter the situation where two--fluid models become
relevant if not inevitable. The classic approach to simplify the
complexity of multi--phase flows and satisfy the engineer's need of
some modeling tools is the well--known volume--averaging method (see
\cite{Ishii1, Prosperetti} for details). Thus, by performing such a
procedure, one can derive a model without surface: a two--fluid
model. More precisely, we denote $\alpha^{\pm}$ by the volume
fraction of the liquid (water) and gas (air), respectively.
Therefore, $\alpha^++\alpha^-=1$. Applying the volume--averaging
procedure to the equations \eqref{1.1} and \eqref{1.2} leads to the
following generic compressible two--fluid model:
\begin{equation}\label{1.4}
\left\{\begin{array}{l}
\partial_{t}\left(\alpha^{\pm} \rho^{\pm}\right)+\operatorname{div}\left(\alpha^{\pm} \rho^{\pm} u^{\pm}\right)=0, \\
\partial_{t}\left(\alpha^{\pm} \rho^{\pm} u^{\pm}\right)+\operatorname{div}\left(\alpha^{\pm} \rho^{\pm} u^{\pm} \otimes u^{\pm}\right)
+\alpha^{\pm} \nabla P=-g\alpha^{\pm}\rho^{\pm} e_3\pm F_D,
\end{array}\right.
\end{equation}
where the two fluids are assumed to share the common pressure $P$.
\par
We have already discussed the case of water waves, where a separated
flow can lead to a two--fluid model from the viewpoint of practical
modeling. As stated before, two--fluid flows are very common in
nature, but also in various industry applications such as nuclear
power, chemical processing, oil and gas manufacturing. In terms of
the context, the models used for simulation may be very different.
However, averaged models share the same structure as \eqref{1.4}. By
introducing viscosity and capillary effects, one can generalize the
above system \eqref{1.4} to
\begin{equation}\label{1.5}
\left\{\begin{array}{l}
\partial_{t}\left(\alpha^{\pm} \rho^{\pm}\right)+\operatorname{div}\left(\alpha^{\pm} \rho^{\pm} u^{\pm}\right)=0, \\
\partial_{t}\left(\alpha^{\pm} \rho^{\pm} u^{\pm}\right)+\operatorname{div}\left(\alpha^{\pm} \rho^{\pm} u^{\pm} \otimes u^{\pm}\right)
+\alpha^{\pm} \nabla P=\operatorname{div}\left(\alpha^{\pm} \tau^{\pm}\right)+\sigma^\pm\alpha^{\pm}
\rho^{\pm}\nabla\Delta(\alpha^{\pm} \rho^{\pm}), \\
P=P^{\pm}\left(\rho^{\pm}\right)=A^{\pm}\left(\rho^{\pm}\right)^{\bar{\gamma}^{\pm}},
\end{array}\right.
\end{equation}
where $\rho^{\pm}(x, t) \geqq 0, u^{\pm}(x, t)$ and
$P^{\pm}\left(\rho^{\pm}\right)=A^{\pm}\left(\rho^{\pm}\right)^{\bar{\gamma}^{\pm}}$
denote the densities, velocities of each phase, and the two pressure
functions, respectively. $\bar{\gamma}^{\pm} \geqq 1, A^{\pm}>0$ are
positive constants. In what follows, we set $A^{+}=A^{-}=1$ without
loss of any generality. Moreover, $\tau^{\pm}$ are the viscous
stress tensors
\begin{equation}\label{1.6}
\tau^{\pm}:=\mu^{\pm}\left(\nabla u^{\pm}+\nabla^{t}
u^{\pm}\right)+\lambda^{\pm} \operatorname{div} u^{\pm} \mathrm{Id},
\end{equation}
where the constants $\mu^{\pm}$ and $\lambda^{\pm}$ are shear and
bulk viscosity coefficients satisfying the physical condition:
$\mu^{\pm}>0$ and $2 \mu^{\pm}+3 \lambda^{\pm} \geqq 0,$ which
implies that $\mu^{\pm}+\lambda^{\pm}>0.$ For more information about
this model, we refer to \cite{Bear, Brennen1, Bresch1, Bresch2,
Evje3, Evje4, Evje8, Evje9, Friis1, Raja, Vasseur, Wen1, Yao2,
Zhang4} and references therein. However, it is well--known that as
far as mathematical analysis of two--fluid model is concerned, there
are many technical challenges. Some of them involve, for example:
\begin{itemize}
\item The corresponding linear system of the model has
multiple eigenvalue, which makes mathematical analysis
(well--posedness and stability) of the model become quite difficult
and complicated;

\item Transition to single--phase regions, i.e, regions where the mass
$\alpha^{+} \rho^{+}$ or $\alpha^{-} \rho^{-}$ becomes zero, may
occur when the volume fractions $\alpha^{\pm}$ or the densities
$\rho^{\pm}$ become zero;

\item The system is non--conservative, since the non--conservative terms $\alpha^{\pm} \nabla
P^{\pm}$ are involved in the momentum equations. This brings various
 mathematical difficulties for us to employ methods used
for single phase models to the two--fluid model.

\end{itemize}
\par
Bresch et al. in the seminal work \cite{Bresch1} considered the
generic two-fluid model \eqref{1.5} with the following special
density-dependent viscosities:
\begin{equation}\label{1.7}
\mu^\pm(\rho^\pm)=\mu^\pm \rho^\pm,~~~~~\lambda^\pm(\rho^\pm)=0.
\end{equation}
They obtained the global weak solutions in periodic domain with
$1<\overline{\gamma}^{\pm}< 6$. However, as indicated by themselves,
 their methods rely heavily on the above special
density-dependent viscosities, and particularly cannot handle the
case of constant viscosity coefficients as in \eqref{1.6}. Later,
Bresch--Huang--Li \cite{Bresch2} established the global existence of
weak solutions in one space dimension without capillary effects
(i.e., $\sigma^\pm=0$) when $\overline{\gamma}^{\pm}>1$ by taking
advantage of the one space dimension. Recently, Cui--Wang--Yao--Zhu
\cite{c1} obtained the time--decay rates of classical solutions for
model \eqref{1.5} with the following special density-dependent
viscosities with equal viscosity coefficients, and equal capillary
coefficients:
\begin{equation}\label{1.8}
\mu^\pm(\rho^\pm)=\nu\rho^\pm,~~~~~\lambda^\pm(\rho^\pm)=0,~~~\sigma^+=\sigma^-=\sigma.
\end{equation}
Based on the above special choice for viscosities and capillary
coefficients, they can take a linear combination of model
\eqref{1.5} to reformulate it
 into two $4\times 4$ systems whose linear parts are decoupled with each other and possess
the same dissipation structure as that of the compressible
Navier--Stokes-Korteweg system, and then employ the similar
arguments as in \cite{Bian, Wang-Tan} to prove their main results.
However, since this reformulation played a crucial role in their
analysis, the case of constant viscosities, even if the equal
constant viscosities (i.e.,
$\mu^\pm(\rho^\pm)=\nu,~~\lambda^\pm(\rho^\pm)=\lambda$), cannot be
handled in their settings.

 \par
  In conclusion, all the works \cite{Bresch1, c1} depend essentially on the special
density-dependent viscosities. Therefore, a natural and important
problem is that what will happen when we consider the general
constant viscosities as in \eqref{1.6}. That is to say, what about
the global well--posedness and large time behavior of Cauchy problem
to the two--fluid model \eqref{1.5}--\eqref{1.6} in high dimensions.
However, to our best knowledge, so far there is no result on
mathematical theory of the two--fluid model \eqref{1.5}--\eqref{1.6}
in high dimensions. \par The main purpose of this work is to
establish global well--posedness and large time behavior of
classical solution to the two--fluid model \eqref{1.5}--\eqref{1.6}.
More precisely, for any integer $\ell\geq3$, we show that the
densities and velocities of model \eqref{1.5}--\eqref{1.6} converge
to their corresponding equilibrium states at the $L^2$ rate
$(1+t)^{-\frac{3}{4}}$, and the $k$($\in [1, \ell]$) order spatial
derivatives of them converge to zero at the $L^2$ rate
$(1+t)^{-\frac{3}{4}-\frac{k}{2}}$, which are the same as ones of
the compressible Navier--Stokes system \cite{Duan, Guo2},
Navier--Stokes--Korteweg system \cite{Bian, Wang-Tan} and heat
equation. Moreover, the linear combination of the fraction densities
($\beta^+\alpha^+\rho^++\beta^-\alpha^-\rho^-$) converges to its
corresponding equilibrium state at the $L^2$ rate
$(1+t)^{-\frac{3}{4}}$, and its $k$($\in [1, \ell]$) order spatial
derivative converges to zero at the $L^2$ rate
$(1+t)^{-\frac{3}{4}-\frac{k}{2}}$, but the fraction densities
($\alpha^\pm\rho^\pm$) themselves converge to their corresponding
equilibrium states at the $L^2$ rate $(1+t)^{-\frac{1}{4}}$, and the
$k$($\in [1, \ell]$) order spatial derivatives of them converge to
zero at the $L^2$ rate $(1+t)^{-\frac{1}{4}-\frac{k}{2}}$, which are
slower than ones of their linear combination
($\beta^+\alpha^+\rho^++\beta^-\alpha^-\rho^-$) and the densities.
We think that this phenomenon should owe to the special structure of
the system. Finally, for well--chosen initial data, we also prove
the lower bounds on the decay rates, which are the same as those of
the upper decay rates. Therefore, these decay rates are optimal for
the compressible two--fluid model.
\subsection{New formulation of system \eqref{1.5} and Main Results}
In this subsection, we devote ourselves to reformulating the system
\eqref{1.5} and stating the main results. The relations between the
pressures of $\eqref{1.5}_3$ implies
\begin{equation}\label{1.9}
{\rm d}P
  =
  s_+^2
 {\rm d}\rho^+
  =
   s_-^2
    {\rm d}\rho^-,
 \quad
  {\rm where}
   \quad
    s_\pm
     :=
      \sqrt{ \frac{{\rm d}P}{{\rm d}\rho^\pm}(\rho^\pm)}.
\end{equation}
Here $s_\pm$ represent the sound speed of each phase respectively.
As in \cite{Bresch1}, we introduce the fraction densities
\begin{equation}\label{1.10}
R^\pm
 =
  \alpha^\pm \rho^\pm,
\end{equation}
which together with the relation $\alpha^++\alpha^-=1$ leads to
\begin{equation}\label{1.11}
{\rm d}
 \rho^+
  =
   \frac{1}{\alpha_+}
    ({\rm d}R^+
       -
        \rho^+
         {\rm d}\alpha^+),
  \quad
   {\rm d}
    \rho^-
     =
      \frac{1}{\alpha_-}
       ({\rm d}R^-
         +
          \rho^-
           {\rm d}\alpha^+).
\end{equation}
From \eqref{1.9}--\eqref{1.10}, we finally get
\begin{equation}\notag
{\rm d}\alpha^+
 =
  \frac{\alpha^-s_+^2}{\alpha^-\rho^+s_+^2+\alpha^+\rho^-s_-^2}
   {\rm d}R^+
    -
     \frac{\alpha^+s_-^2}{\alpha^-\rho^+s_+^2+\alpha^+\rho^-s_-^2}
   {\rm d}R^-.
\end{equation}
Substituting the above equality into \eqref{1.11} yields
\begin{equation}\notag
{\rm d}\rho^+
 =
  \frac{s_-^2}{\alpha^-\rho^+s_+^2+\alpha^+\rho^-s_-^2}
   (\rho^-{\rm d}R^++\rho^+{\rm d}R^-),
\end{equation}
and
\begin{equation}\notag
{\rm d}\rho^-
 =
  \frac{s_+^2}{\alpha^-\rho^+s_+^2+\alpha^+\rho^-s_-^2}
   (\rho^-{\rm d}R^++\rho^+{\rm d}R^-),
\end{equation}
which combined with $(\ref{1.9})$ imply for the pressure
differential ${\rm d}P$
\begin{equation}\label{1.12}
{\rm d}P
 =
  \mathcal{C}^2(\rho^- {\rm d} R^+ +\rho^+ {\rm d}R^-),
\end{equation}
where
\begin{equation}\notag
\mathcal{C}^2
 :=
   \frac{s_+^2s_-^2}{\alpha^-\rho^+s_+^2+\alpha^+\rho^-s_-^2},
    \quad
     {\rm and}
      \quad
       s_\pm^2
        =
        \frac{{\rm d}P(\rho^\pm)}{{\rm d}\rho^\pm}
         =
          \tilde{\gamma}^\pm \frac{P(\rho^\pm)}{\rho^\pm}.
\end{equation}
Next, by using the relation: $\alpha^+ +\alpha^- =1$ again, we can
get
\begin{equation}\label{1.13}
\frac{R^+}{\rho^+}
 +
  \frac{R^-}{\rho^-}
   =1,
    \quad
     {\rm and \ therefore}
      \quad
       \rho^-
        =
         \frac{R^-\rho^+}{\rho^+-R^+}.
\end{equation}
By virtue of $\eqref{1.5}_3$, we have
\begin{equation}\notag
\varphi(\rho^+, R^+, R^-)
 :=
  P(\rho^+)
   -
    P\left(\frac{R^-\rho^+}{\rho^+-R^+}\right)
     =
      0.
\end{equation}
Consequently, for any given two positive constants $\tilde R^+$,
$\tilde R^-$, there exists $\tilde \rho^+>\tilde R^+$ such that
\begin{equation}\notag
\varphi(\tilde \rho^+, \tilde R^+, \tilde R^-)=0.
\end{equation}
Differentiating $\varphi$ with respect to $\rho^+$, we get
\begin{equation}\notag
\frac{\partial\varphi}{\partial\rho^+}(\rho^+, R^+, R^-)
 =
  s_+^2+
   s_-^2
    \frac{R^-R^+}{(\rho^+-R^+)^2},
\end{equation}
which implies
\begin{equation}\notag
\frac{\partial\varphi}{\partial\rho^+}(\tilde\rho^+, \tilde R^+,
\tilde R^-)>0.
\end{equation}
Thus, this together with Implicit Function Theorem and \eqref{1.10},
\eqref{1.13} implies that the unknowns $\rho^\pm$, $\alpha^\pm$ and
$\mathcal{C}$ can be given by
\begin{equation}\notag
\rho^\pm
     =
      \varrho^\pm(R^+,R^-),
\qquad
      \alpha^\pm
       =
        \alpha^\pm(R^+,R^-),
         \quad
     {\rm and \ therefore}
      \quad
      \mathcal{C}=\mathcal{C}(R^+, R^-).
\end{equation}
 We refer to [\cite{Bresch1}, pp. 614] for the details.
\par
Therefore, we can rewrite system \eqref{1.5} into the following
equivalent form:
\begin{equation}\label{1.14}
\left\{\begin{array}{l}
\partial_{t} R^{\pm}+\operatorname{div}\left(R^{\pm} u^{\pm}\right)=0, \\
\partial_{t}\left(R^{+} u^{+}\right)+\operatorname{div}\left(R^{+} u^{+}
\otimes u^{+}\right)+\alpha^{+} \mathcal{C}^{2}\left[\rho^{-} \nabla
R^{+}+\rho^{+}\nabla R^{-}\right] \\
\hspace{1.8cm}=\operatorname{div}\left\{\alpha^{+}\left[\mu^{+}\left(\nabla
u^{+}+\nabla^{t} u^{+}\right)
+\lambda^{+} \operatorname{div} u^{+} \operatorname{Id}\right]\right\}+\sigma^+R^+\nabla\Delta R^+, \\
\partial_{t}\left(R^{-} u^{-}\right)+\operatorname{div}\left(R^{-} u^{-} \otimes u^{-}\right)+\alpha^{-}
\mathcal{C}^{2}\left[\rho^{-} \nabla R^{+}+\rho^{+}\nabla R^{-}\right] \\
\hspace{1.8cm}=\operatorname{div}\left\{\alpha^{-}\left[\mu^{-}\left(\nabla
u^{-}+\nabla^{t} u^{-}\right)+\lambda^{-} \operatorname{div} u^{-}
\operatorname{Id}\right]\right\}+\sigma^-R^-\nabla\Delta R^-.
\end{array}\right.
\end{equation}
In the present paper, we consider the Cauchy problem of \eqref{1.14}
subject to the initial condition
\begin{equation}\label{1.15} (R^{+}, u^{+}, R^-, u^{-})(x,
0)=(R_{0}^{+}, u_{0}^{+}, R_{0}^+, u_{0}^{-})(x)\rightarrow(\bar
R^{+}, \overrightarrow{0}, \bar R^{-}, \overrightarrow{0}) \quad
\hbox{as}\quad |x|\rightarrow\infty \in \mathbb{R}^{3},
\end{equation}
where two positive constants  $\bar{R}^{+}$ and $\bar{R}^{-}$ denote
the background doping profile, and in the present paper are taken as
1 for simplicity.

\bigskip
\par
 Before stating
our main result, let us first introduce the notations and
conventions used throughout this paper. We use $H^k(\mathbb R^3)$ to
denote the usual Sobolev space with norm $\|\cdot\|_{H^k}$ and
$L^p$, $1\leq p\leq \infty$ to denote the usual $L^p(\mathbb R^3)$
space with norm $\|\cdot\|_{L^p}$.  For the sake of conciseness, we
do not precise in functional space names when they are concerned
with scalar-valued or vector-valued functions, $\|(f, g)\|_X$
denotes $\|f\|_X+\|g\|_X$.  We will employ the notation $a\lesssim
b$ to mean that $a\leq Cb$ for a universal constant $C>0$ that only
depends on the parameters coming from the problem. We denote
$\nabla=\partial_x=(\partial_1,\partial_2,\partial_3)$, where
$\partial_i=\partial_{x_i}$, $\nabla_i=\partial_i$ and put
$\partial_x^\ell f=\nabla^\ell f=\nabla(\nabla^{\ell-1}f)$.  Let
$\Lambda^s$ be the pseudo differential operator defined by
\begin{equation}\Lambda^sf=\mathfrak{F}^{-1}(|{\bf \xi}|^s\widehat f),~\hbox{for}~s\in \mathbb{R},\nonumber\end{equation}
where $\widehat f$ and $\mathfrak{F}(f)$ are the Fourier transform
of $f$. The homogenous Sobolev space $\dot{H}^s(\mathbb{R}^3)$ with
norm given by $\|f\|_{\dot{H}^s}\overset{\triangle}=\|\Lambda^s
f\|_{L^2}$. For a radial function $\phi\in C_0^\infty(\mathbb
R^3_{{\bf \xi}})$  such that $\phi({\bf \xi})=1$ when $|{\bf
\xi}|\leq \frac{\eta}{2}$ and $\phi({\bf \xi})=0$ when $|{\bf
\xi}|\geq \eta$, where $\eta$ is defined in Lemma \ref{2.1}, we
define the low--frequency part of $f$ by
$$f^l=\mathfrak{F}^{-1}[\phi({\bf \xi})\widehat f]$$
and the high--frequency part of $f$ by
$$f^h=\mathfrak{F}^{-1}[(1-\phi({\bf \xi}))\widehat f].$$
It is direct to check that $f=f^l+f^h$ if Fourier transform of $f$
exists.

\medskip
Now, we are in a position to state our main result.
\smallskip
\begin{Theorem}\label{1mainth}Assume that $R_{0}^{+}-1,~ R_{0}^{-}-1\in H^{\ell+1}(\mathbb{R}^3)$
 and $u_{0}^{+},~ u_{0}^{-}\in H^\ell(\mathbb{R}^3)$ for an integer
$\ell\geq 3$. Then there exists a constant $\delta_0$ such that if
\begin{equation}\label{1.16}
K_0:=\left\|\left(R_{0}^{+}-1, R_{0}^{-}-1,
\right)\right\|_{H^{\ell+1}\cap L^1}+\left\|\left(u_{0}^{+},
u_{0}^{-}\right)\right\|_{H^{\ell}\cap L^1} \leq \delta_0,
\end{equation}
then the Cauchy problem \eqref{1.14}--\eqref{1.15} admits a unique
solution $\left(R^{+}, u^{+}, R^{-}, u^{-}\right)$ globally in time
in the sense that \[
\begin{array}{l}
R^{+}-1, R^{-}-1 \in C^{0}\left([0, \infty) ;
H^{\ell+1}\left(\mathbb{R}^{3}\right)\right) \cap C^{1}\left([0,
\infty) ;
H^{\ell}\left(\mathbb{R}^{3}\right)\right), \\
u^{+}, u^{-} \in C^{0}\left([0, \infty) ;
H^{\ell}\left(\mathbb{R}^{3}\right)\right) \cap C^{1}\left([0,
\infty) ; H^{\ell-2}\left(\mathbb{R}^{3}\right)\right),
\end{array}
\]
and satisfies
\begin{align}\begin{split}\label{1.17}
&\|(R^+-1,
R^--1)(t)\|_{H^{\ell+1}}^2+\|(u^+,u^-)(t)\|_{H^{\ell}}^2\\
&\quad+
\|\left[\beta^+(R^+-1)+\beta^-(R^--1)\right](t)\|_{H^{\ell}}^2\displaystyle+\int_0^t\left(\|\nabla(R^+-1,
R^--1)(\tau)\|_{H^{\ell}}^2+\|(u^+,u^-)(\tau)\|_{H^{\ell}}^2\right.\\
&\quad+\left.\|\left[\beta^+(R^+-1)+\beta^-(R^--1)\right](\tau)\|_{H^{\ell}}^2\right)\textrm{d}\tau\leq
CK_0^2.\end{split}\end{align}
 Moreover, the following convergence rates hold true.
\smallskip
\begin{itemize}
\item {\bf Upper bounds.} For any $t\geq 0,$ and $0\leq k\leq \ell$, it holds that
\begin{equation}\label{1.18}\left\|\nabla^k\left(\rho^{+}-\bar{\rho}^+, \rho^{-}-\bar{\rho}^-\right)(t)\right\|_{H^{\ell-k}}\leq
CK_0(1+t)^{-\frac{3}{4}-\frac{k}{2}},
\end{equation}
\begin{equation}\label{1.19}\left\|\nabla^{k}\left(u^{+},
u^{-}\right)(t)\right\|_{H^{\ell-k}} \leq
CK_0(1+t)^{-\frac{3}{4}-\frac{k}{2}},
\end{equation}
\begin{equation}\label{1.20}\left\|\nabla^{k}\left[\beta^+(R^+-1)+\beta^-(R^--1)\right](t)\right\|_{H^{\ell-k}} \leq
CK_0(1+t)^{-\frac{3}{4}-\frac{k}{2}},
\end{equation}
\begin{equation}\label{1.21}\left\|\nabla^k\left(R^{+}-1, R^{-}-1,
\right)(t)\right\|_{H^{\ell-k+1}}\leq
CK_0(1+t)^{-\frac{1}{4}-\frac{k}{2}},
\end{equation}
where $\bar{\rho}^\pm=\rho^\pm(1,1)$ denote equilibrium states of
$\rho^\pm$ respectively, and
$\beta^\pm=\sqrt{\frac{\bar{\rho}^\mp}{{\rho}^\pm}}$.
\smallskip
\item {\bf Lower bounds.} Let  $(n_0^+,u_0^+,R_0^-,u_0^-)=(R_0^+-1,
u_0^+,R_0^--1,u_0^-)$ and assume that Fourier transform of functions
$(n_0^+,u_0^+,n_0^-,u_0^-)$ satisfy
\begin{equation}\label{1.22}\quad\quad\widehat{n}_0^-(\xi)=0,~
\wedge^{-1}\text{\rm div}\
\widehat{u}_0^+(\xi)=\widehat{n}_0^+(\xi)=0,~\text{\rm and }~
\wedge^{-1}\text{\rm div}\ \widehat{u}_0^-(\xi)-K_0^\vartheta\sim
|\xi|^s,
\end{equation}
for any $|\xi|\le \eta$, where $\vartheta<2$ and $s>0$ are two given
constants. Then there is a positive constant $C_0$ independent of
$t$ such that for any large enough $t$ and $0\leq k\leq \ell$, it
holds that
\begin{equation}\label{1.23}
\min\left\{\|\nabla^k
(\rho^+-\bar{\rho}^+)(t)\|_{H^{\ell-k}},\|\nabla^k
(\rho^--\bar{\rho}^-)(t)\|_{H^{\ell-k}}\right\}\geq
C_0(1+t)^{-\frac{3}{4}-\frac{k}{2}},
\end{equation}
\begin{equation}\label{1.24}
\min\left\{\|\nabla^k u^+(t)\|_{H^{\ell-k}},\|\nabla^k
u^-(t)\|_{H^{\ell-k}}\right\}\geq
C_0(1+t)^{-\frac{3}{4}-\frac{k}{2}},
\end{equation}
\begin{equation}\label{1.25}
\min\left\{\|\nabla^k
\left[\beta^+(R^+-1)+\beta^-(R^--1)\right](t)\|_{H^{\ell-k}}\right\}\geq
C_0(1+t)^{-\frac{3}{4}-\frac{k}{2}},
\end{equation}
\begin{equation}\label{1.26}
\min\left\{\|\nabla^k (R^+-1)(t)\|_{H^{\ell-k+1}},\|\nabla^k
(R^--1)(t)\|_{H^{\ell-k+1}}\right\}\geq
C_0(1+t)^{-\frac{1}{4}-\frac{k}{2}},
\end{equation}
\end{itemize}
\end{Theorem}

\begin{remark} Compared to Cui--Wang--Yao--Zhu
\cite{c1}, where the model \eqref{1.5} with \eqref{1.8} was
considered, the main new contribution of Theorem \ref{1mainth} is
four--fold: First, as mentioned before, \eqref{1.8} played an
essential role in their analysis. Therefore, we need to develop new
thoughts to overcome the difficulties arising from the general
constant viscosities as in \eqref{1.6}, which will be explained
later. Second, \eqref{1.18} and \eqref{1.20} give the optimal decay
rates on the densities and linear combination of the fraction
densities: $\beta^+(R^+-1)+\beta^-(R^--1)$, which are totally new as
compared to Cui--Wang--Yao--Zhu \cite{c1}. Third, noticing that in
Cui--Wang--Yao--Zhu \cite{c1}, the $(\ell-1)$--th and $\ell$--th
spatial derivatives of the velocities decay at the same $L^2$ rate
$(1+t)^{-\frac{3}{4}-\frac{\ell-2}{2}}$, and the $\ell$--th and
$(\ell+1)$--th spatial derivatives of the fraction densities decay
at the same $L^2$ rate $(1+t)^{-\frac{3}{4}-\frac{\ell-1}{2}}$,
which are slower than ones in \eqref{1.19} and \eqref{1.21}.
Finally, for well--chosen initial data, \eqref{1.23}--\eqref{1.26}
show the lower bounds on the decay rates, which are the same as
those of the upper decay rates, and completely new as compared to
Cui--Wang--Yao--Zhu \cite{c1}. Therefore, our decay rates are
optimal in this sense.
\end{remark}

\smallskip
\smallskip

\indent Now, let us sketch the strategy of proving Theorem
\ref{1mainth} and explain some of the main difficulties and
techniques involved in the process. Different from
Cui--Wang--Yao--Zhu \cite{c1} where the model \eqref{1.5} with
\eqref{1.8} was considered, we need to develop new ideas to tackle
with the difficulties from general constant viscosities and
capillary coefficients. To see this, by taking $n^{\pm}=R^{\pm}-1$,
one can write the corresponding linear system of model \eqref{1.5}
in terms of the variables $(n^+, u^+, n^-, u^-)$:
\begin{equation}\label{1.27}
\left\{\begin{array}{l}
\partial_{t} n^++\operatorname{div}u^+=0, \\
\partial_{t}u^{+}+\beta_1\nabla n^++\beta_2\nabla
n^--\nu^+_1\Delta u^+-\nu^+_2\nabla\operatorname{div} u^+-\sigma^+\nabla\Delta n^+=0, \\
\partial_{t} n^-+\operatorname{div}u^-=0, \\
\partial_{t}u^{-}+\beta_3\nabla n^++\beta_4\nabla
n^--\nu^-_1\Delta u^--\nu^-_2\nabla\operatorname{div} u^--\sigma^-\nabla\Delta n^-=0, \\
\end{array}\right.
\end{equation}
where $\nu_{1}^{\pm}=\frac{\mu^{\pm}}{\bar{\rho}^{\pm}}$,
$\nu_{2}^{\pm}=\frac{\mu^{\pm}+\lambda^{\pm}}{\bar{\rho}^{\pm}}>0$,
$\beta_{1}=\frac{\mathcal{C}^{2}(1,1)
\bar{\rho}^{-}}{\bar{\rho}^{+}}$,
$\beta_{2}=\beta_{3}=\mathcal{C}^{2}(1,1)$,
$\beta_{4}=\frac{\mathcal{C}^{2}(1,1)
\bar{\rho}^{+}}{\bar{\rho}^{-}}$. In view of \eqref{1.8}, the system
\eqref{1.27} can be reduced to
\begin{equation}\label{1.28}
\left\{\begin{array}{l}
\partial_{t} n^++\operatorname{div}u^+=0, \\
\partial_{t}u^{+}+\beta_1\nabla n^++\beta_2\nabla
n^--\nu\Delta u^+-\nu\nabla\operatorname{div} u^+-\sigma\nabla\Delta n^+=0, \\
\partial_{t} n^-+\operatorname{div}u^-=0, \\
\partial_{t}u^{-}+\beta_3\nabla n^++\beta_4\nabla
n^--\nu \Delta u^--\nu\nabla\operatorname{div} u^--\sigma\nabla\Delta n^-=0. \\
\end{array}\right.
\end{equation}
Based on the above special linear system, the main observation of
Cui--Wang--Yao--Zhu \cite{c1} is to introduce four linear
combinations:
\begin{equation}\notag
N^+:=\beta_3 n^++\beta_4 n^-,\ \ \ \ N^-:=\beta_3 n^+-\beta_4 n^-,
\end{equation}
\begin{equation}\notag
U^+:= \beta_3 u^++ \beta_4 u^-,\ \ \ \ U^-:= \beta_3  u^+- \beta_4
u^-.
\end{equation}
Then, the system \eqref{1.28} can be divided into two new linear
system:
\begin{equation}\label{1.29}
\left\{\begin{array}{l}
\partial_t N^+
  +
    {\rm div}U^+
       =0,\\[2mm]
\partial_tU^+
 +
  ( \beta_1 + \beta_4 )\nabla N^+
     -
      \nu\Delta U^+
       -
        \nu\nabla{\rm div}U^+
         -
          \sigma \nabla \Delta N^+
           =
            0,\\[2mm]
 \end{array}
        \right.
\end{equation}
and
\begin{equation}\label{1.30}
\left\{\begin{array}{l}
\partial_t N^-
     +
    {\rm div}U^-
       =0,\\[2mm]
\partial_t U^-
 +
  (\beta_1 - \beta_4)\nabla N^+
     -
      \nu\Delta U^-
       -
        \nu\nabla{\rm div}U^-
         -
          \sigma \nabla \Delta N^-
           =
            0\\
 \end{array}
        \right.
\end{equation}
It is worth mentioning that \eqref{1.29} are decoupled from $(N^-,
U^-)$ and possesses the same dissipative structure as that of the
compressible Navier--Stokes--Korteweg system \cite{Bian, Wang-Tan},
while \eqref{1.30} are coupled with $\nabla N^+$ and possesses the
similiar dissipative structure as that of the compressible
Navier--Stokes--Korteweg system \cite{Bian, Wang-Tan}. Thus, by
making full use of these good properties of \eqref{1.29} and
\eqref{1.30},
 Cui--Wang--Yao--Zhu \cite{c1} can modify the methods of \cite{Bian,
Wang-Tan} to prove their main results. However, since their analysis
relies heavily on this reformulation, the case of general constant
viscosities as in \eqref{1.6}, even if the equal constant
viscosities, cannot be handled in their settings. Indeed, for the
case of the equal constant viscosities (i.e.
$\mu^\pm(\rho^\pm)=\nu,~~~\lambda^\pm(\rho^\pm)=\lambda$), we have
\begin{equation}\notag
\nu_1^\pm=\frac{\nu}{\bar{\rho}^\pm},~~~\hbox{and}~~~{\nu_2^\pm=\frac{\nu+\lambda}{\bar{\rho}^\pm}}.
\end{equation}
This particularly implies that the system \eqref{1.27} cannot be
reduced to system \eqref{1.28}, since $\nu_1^+\neq \nu_1^-$ and
$\nu_2^+\neq \nu_2^-$ when $\bar{\rho}^+\neq \bar{\rho}^+$. The key
idea here is that, instead of using this reformulation, we will work
on the model \eqref{1.5} with \eqref{1.6} directly, which makes the
problem become quite difficult and complicated. In what follows, we
will give a brief interpretation for the main idea of the proof.\par
To begin with, we give a heuristic description of our strategy.
Multiplying $\eqref{1.27}_1$, $\eqref{1.27}_2$, $\eqref{1.27}_3$ and
$\eqref{1.27}_4$ by $\frac{{\beta}_1}{{\beta}_2}n^+$,
$\frac{1}{{\beta}_2}u^+$, $\frac{{\beta}_4}{{\beta}_3}n^-$ and
$\frac{1}{{\beta}_3}u^-$, one can easily get the nature energy
equation of the linear system \eqref{1.27}:
\begin{align}\begin{split}\label{1.31}
&\partial_t\mathcal{E}_0(t)+\mathcal{D}_0(t)\\
&:=\displaystyle\partial_t\int_{\mathbb{R}^3}\left(
\frac{{\beta}_1}{2{\beta}_2}\left|n^+\right|^2+\frac{{\beta}_4}{2{\beta}_3}\left|n^-\right|^2+n^+n^-
+\frac{\sigma^+}{2{\beta}_2}\left|\nabla
n^+\right|^2+\frac{\sigma^-}{2{\beta}_3}\left|\nabla n^-\right|^2 +
\frac{1}{2{\beta}_2}\left|u^+\right|^2+\frac{1}{2{\beta}_3}\left|u^-\right|^2\right)\textrm{d}x\\
&\quad+\displaystyle\int_{\mathbb{R}^3}\frac{1}{{\beta}_2}\left(
\nu_1^+\left|\nabla
u^+\right|^2+\nu_2^+\left|\hbox{div}u^+\right|^2\right)+\frac{1}{{\beta}_3}\left(
\nu_1^-\left|\nabla
u^-\right|^2+\nu_2^-\left|\hbox{div}u^-\right|^2\right)\textrm{d}x=0,
\end{split}\end{align}
where $\mathcal{E}_0(t)$ and $\mathcal{D}_0(t)$ denote the nature
energy and dissipation, respectively. Noticing the fact that
${\beta}_1{\beta}_4={\beta}_2{\beta}_3={\beta}_3^2$, it is clear
that
\begin{equation}\label{1.32}\displaystyle\mathcal{E}_0(t)=\frac{1}{2}\int_{\mathbb{R}^3}\left(
\left(\beta^+n^++\beta^-n^-\right)^2+\frac{\sigma^+}{{\beta}_2}\left|\nabla
n^+\right|^2+\frac{\sigma^-}{{\beta}_3}\left|\nabla n^-\right|^2+
\frac{1}{{\beta}_2}\left|u^+\right|^2+\frac{1}{{\beta}_3}\left|u^-\right|^2\right)\textrm{d}x.\end{equation}
This together with the energy equation \eqref{1.31} makes it
impossible for us to get the uniform energy estimates of $n^\pm$
simultaneously, even though in the linear level, but possibly the
uniform energy estimates of their linear combination:
$\beta^+n^++\beta^-n^-$, and $\nabla n^\pm$. Therefore, the linear
combination: $\beta^+n^++\beta^-n^-$ may be a good dissipative
variable. On the other hand, by virtue of Mean Value Theorem, we
have $\rho^\pm-\bar{\rho}^\pm \sim
\frac{\mathcal{C}^2(1,1)\sqrt{\bar{\rho}^+\bar{\rho}^-}}{s^2_{\pm}}\left(\beta^+n^++\beta^-n^-\right)$.
In the spirit of these heuristic observations, it is natural to
conjecture that $\rho^\pm-\bar{\rho}^\pm$ have the same decay rate
in time as the density of the compressible Navier--Stokes--Korteweg
system \cite{Bian, Wang-Tan}. As a matter of fact, this key
observation plays a vital role in our analysis. Roughly speaking,
our proofs mainly involves the following four steps.\par First, we
deduce spectral analysis and linear $L^2$ estimates on
 $(n^+, u^+,n^-, u^-)$ of the solution to the
linear system of \eqref{2.1}. To derive time--decay estimates of the
linear system \eqref{2.11}, it requires us to make a detailed
analysis on the properties of the semigroup. We therefore encounter
a fundamental obstacle that the matrix $\mathcal A(\xi)$ in
\eqref{2.12} is an $8$--order matrix and is not self--adjoint.
Particularly, it is easy to check that the matrix $\mathcal A(\xi)$
cannot be diagonalizable (see \cite{Sideris} pp.807 for example).
Thus, it seems impossible to apply the usual time decay
investigation through spectral analysis. To get around this
difficulty, we will employ the Hodge decomposition technique firstly
introduced by Danchin \cite{Dan1} to split the linear system into
three systems. One is a $4\times 4$ system and the other two are
classic heat equations. Unfortunately, Green Matrix $\mathcal
A_1(\xi)$ for Fourier transform of the $4\times 4$ system may have
multiple eigenvalue, and particularly cannot be diagonalizable. To
overcome this difficulty, we first deal with the case of no multiple
eigenvalue. Then, in the spirit of the case of no multiple
eigenvalue, we handle the case of multiple eigenvalue. The key idea
here is that we first assume
 that a similar expression of the semigroup as
 in \eqref{2.30} holds, which however is crucial for the derivation of
time--decay estimates. Then, by employ a clever decomposition and
careful analysis, we can get the explicit expressions of $P_i(\xi)$
for $i=1,2, 3, 4.$ We refer to the proof of \eqref{2.30} for
details.\par Second, we make energy estimates of the nonlinear
system \eqref{2.1}. To begin with, similar to the proof of
\eqref{1.31}, for $0\leq k\leq \ell$, we have
\begin{align}\label{1.33}
&
 \frac{\rm d}{{\rm d}t}\left\{\|\nabla^k\left(\beta^+ n^++\beta^-n^-\right)\|_{L^2}^2+\frac{\sigma^+}{\beta_2}
 \|\nabla^{k+1}
 n^+\|_{L^2}^2+\frac{\sigma^-}{\beta_3} \|\nabla^{k+1} n^-\|_{L^2}^2
 + \frac{1}{\beta_2} \|\nabla^ku^+\|_{L^2}^2+\frac{1}{\beta_3} \|\nabla^ku^-\|_{L^2}^2
\right\}\notag\\
&\quad+\displaystyle C\Big(\nu^+_1\|\nabla^{k+1}
u^+\|_{L^2}^2+\nu^+_2\|\nabla^k\text{\rm div}
u^+\|_{L^2}^2+\nu^-_1\|\nabla^{k+1}
u^-\|_{L^2}^2+\nu^-_2\|\nabla^k\text{\rm div} u^-\|_{L^2}^2\Big)
\\
&\lesssim\delta\Big(\|\nabla^{k+1} (n^+, n^-)\|_{H^1}^2
+\|\nabla^{k} (u^+, u^-)\|_{L^2}^2\Big).\notag
\end{align}
Noticing that \eqref{1.33} only involves the dissipation  of
$u^\pm$, we need to derive the dissipation estimates of $n^\pm$. In
fact, for $0\leq k\leq \ell$, it holds that
\begin{equation}\label{1.34}
\begin{split}
& \frac{\rm d}{{\rm d}t}
 \left\{\left\langle \nabla^k u^+, \frac{1}{\beta_2}\nabla \nabla^k n^+\right\rangle
 +\left\langle\nabla^k u^-,\frac{1}{\beta_3}\nabla\nabla^k n^-\right\rangle
 \right\}\\
 &\quad+C\Big\|\nabla^{k+1}\left(\beta^+\nabla n^+ +\beta^-\nabla n^-\right)\|_{L^2}^2
+\|\nabla^{k+2} n^+\|_{L^2}^2+\|\nabla^{k+2} n^-\|_{L^2}^2\Big)\\
&\lesssim\Big(\delta \|\nabla^{k+1}
(n^+,n^-)\|_{L^2}^2+\|\nabla^{k+1} (u^+,u^-)\|_{L^2}^2\Big).
\end{split}
\end{equation}
It should be mentioned that different from Lemma \ref{2.2} and Lemma
\ref{2.3} of \cite{c1}, where $\|\nabla^{k} (n^+, n^-)\|_{H^2}^2$
are involved in the right--hand side of the corresponding energy
inequality, \eqref{1.33} and \eqref{1.34} are new and different
since their right--hand side only includes $\|\nabla^{k+1} (n^+,
n^-)\|_{H^1}^2$, and particularly excludes the term $\|\nabla^{k}
(n^+, n^-)\|_{L^2}^2$. These new types of energy inequality
\eqref{1.33} and \eqref{1.34} are crucial for us to close energy
estimates at each $k$--th level.\par Third, we close energy
estimates and prove the upper bounds on optimal decay rates. As
mentioned before, the nature energy equation \eqref{1.31} implies
that it seems impossible to get the energy estimates of $n^\pm$, but
possibly the linear combination $\beta^+n^++\beta^-n^-$ and the
derivatives of $n^\pm$. Inspired by this key observation, the main
idea here is to introduce two new time--weighted energy functionals
and then estimate them separately. To begin with, we define the
following two time--weighted energy functionals
\begin{equation}\label{1.35} E_k^{\ell}(t)=\sup\limits_{0\leq\tau\leq t}\Big
\{(1+\tau)^{\frac{3}{4}+\frac{k}{2}}\Big(\|\nabla^k(\beta^+n^++\beta^-n^-,u^+,u^-)(\tau)\|_{H^{\ell-k}}+
\|\nabla^{k+1}(n^+, n^-)(\tau)\|_{H^{\ell-k}}\Big)\Big \},
\end{equation}
for $0\leq k\leq \ell$, and
\begin{equation}\label{1.36} E_0(t)=\sup\limits_{0\leq\tau\leq t}\Big
\{(1+\tau)^{\frac{1}{4}} \|(n^+, n^-)(\tau)\|_{L^2}\Big \}.
\end{equation}
Then, by virtue of \eqref{1.33} and \eqref{1.34}, we have
 \begin{equation}\label{1.37}\frac{\mathrm{d}}{\mathrm{d}t}\mathcal
E_0^{\ell}(t)+C\mathcal E_0^{\ell}(t)\le
 C\Big(\|(\beta^+n^{+,l}+\beta^-n^{-,l})(t)\|^2_{L^2}+\|\nabla( n^{+,l}, n^{-,l})(t)\|^2_{L^2}+\|( u^{+,l}, u^{-,l})(t)\|^2_{L^2}
\Big),
\end{equation}
where $\mathcal E_0^{\ell}(t)$ is equivalent to
$\|(\beta^+n^++\beta^-n^-)(t)\|^2_{H^{\ell}}+\|\nabla(n^+,
n^-)(t)\|^2_{H^{\ell}}+\|(u^+, u^-)(t)\|^2_{H^{\ell}}$. Next, we
estimate the terms in the right--side of \eqref{1.37}. To illustrate
our idea, we use $\|( u^{+,l}, u^{-,l})(t)\|_{L^2}$ as an example.
By virtue of
 Duhamel's principle, integration by parts,
 linear estimates obtained in Step 1, we have
\begin{equation}\begin{split}\label{1.38}
\|(u^{+,l}, u^{-,l})\|_{L^2}&\leq K_0(1+t)^{-\frac{3}{4}}\|
U(0)\|_{L^1}+\int_0^t(1+t-\tau)^{-\frac{5}{4}}
 \|\mathcal(n^+u^+, n^-u^-)(\tau)\|_{L^1}\mathrm{d}\tau\\
 &\quad+\int_0^t
 \|\text{e}^{(t-\tau)\mathcal{B}}(F_2, F_4)(\tau)\|_{L^2}\mathrm{d}\tau\\
&\lesssim
(1+t)^{-\frac{3}{4}}\left(K_0+E_0(t){E}_0^\ell(t)\right)+\int_0^t
 \|\text{e}^{(t-\tau)\mathcal{B}}(F_2, F_4)(\tau)\|_{L^2}\mathrm{d}\tau.
\end{split}\end{equation} On the other hand, it is clear that the strongly coupling terms
like $g_+\left(n^{+}, n^{-}\right)
\partial_{i} n^{+}+\bar{g}_{+}\left(n^{+}, n^{-}\right) \partial_{i}
n^{-}$ in \eqref{2.3} and $g_-\left(n^{+}, n^{-}\right)
\partial_{i} n^{-}+\bar{g}_{-}\left(n^{+}, n^{-}\right) \partial_{i}
n^{+}$ in \eqref{2.5} devote the slowest time--decay rates to the
third term on the right--side of \eqref{1.38}. Therefore, this
together with \eqref{1.38} implies that we can only obtain the
following estimate:
\begin{equation}\notag
\|(u^{+,l}, u^{-,l})\|_{L^2}\lesssim
(1+t)^{-\frac{3}{4}}\ln(1+t)\left(K_0+E_0(t){E}_0^\ell(t)\right),
\end{equation}
which however is not quickly enough for us to close energy
estimates. To overcome this essential difficulty, it is crucial to
develop new thoughts to deal with the trouble terms:
$g_+\left(n^{+}, n^{-}\right)
\partial_{i} n^{+}+\bar{g}_{+}\left(n^{+}, n^{-}\right) \partial_{i}
n^{-}$ and $g_-\left(n^{+}, n^{-}\right)
\partial_{i} n^{-}+\bar{g}_{-}\left(n^{+}, n^{-}\right) \partial_{i}
n^{+}$. The main idea here is that in view of the linear combination
$\beta^+n^++\beta^-n^-$, we rewrite them in a clever way. More
specifically, we surprising find that
\begin{align*}
g_+\left(n^{+}, n^{-}\right)
\partial_{i} n^{+}+\bar{g}_{+}\left(n^{+}, n^{-}\right) \partial_{i}
n^{-}=\partial_i G_2+\hbox{good terms},\\
g_-\left(n^{+}, n^{-}\right)
\partial_{i} n^{-}+\bar{g}_{-}\left(n^{+}, n^{-}\right) \partial_{i}
n^{+}=\partial_i G_4+\hbox{good terms}.
\end{align*}
With this key observation, one can shift the derivatives of $G_2$
and $G_4$ onto the solution semigroup to derive the desired decay
estimates of $\|(u^{+,l}, u^{-,l})\|_{L^2}$. It is worth mentioning
that good dissipation properties of $\rho^\pm-\bar{\rho}^\pm$ play
an important role in this process. Therefore, we finally deduce that
 \begin{equation}\label{1.39} E_0^{\ell}(t)\lesssim\left[K_0+E_0(t)E_0^\ell(t)+\left(E_0^\ell(t)\right)^2\right].\end{equation}
Next, we deal with $E_0(t)$. As mentioned before, due to special
structure of the natural energy equation \eqref{1.31}, we cannot
employ the nonlinear energy estimates obtained in Step 2. Instead,
we will make full use of the benefit of low-frequency and
high-frequency decomposition to get
\begin{equation}\begin{split}\nonumber\|(n^+, n^-)\|_{L^2}&\lesssim
\left(\|(n^{+,l}, n^{-,l})\|_{L^2}+\|(n^{+,h},
n^{-,h})\|_{L^2}\right)\\
&\lesssim\left(\|(n^{+,l}, n^{-,l})\|_{L^2}+\|\nabla(n^{+},
n^{-})\|_{L^2}\right)\\
&\lesssim
(1+t)^{-\frac{1}{4}}\left[K_0+E_0(t)E_0^\ell(t)+\left(E_0^\ell(t)\right)^2\right],
 \end{split}\end{equation}
which leads to
\begin{equation}\label{1.40}E_0(t)\lesssim\left[K_0+E_0(t)E_0^\ell(t)+\left(E_0^\ell(t)\right)^2\right].\end{equation}
Finally, combining \eqref{1.39} with \eqref{1.40}, we can close
energy estimates (see the proof of \eqref{4.19} for details). For
$1\leq k\leq \ell$, the new difficulty is how to control the terms
involving $\nabla^{\ell+2}u^\pm$ which however don't belong to the
solution space. To get around this difficulty, we separate the time
interval into two parts and make full use of the benefit of
low-frequency and high-frequency decomposition (see the proof of
\eqref{4.31} for details). Then, we can employ similar arguments to
prove $E_k^\ell(t)\leq CK_0$ by induction. This in turn proves the
optimal decay rates in \eqref{1.19}--\eqref{1.21}. Finally, we can
make full use of the expression of the pressure differential in
\eqref{1.12}, and \eqref{1.19}--\eqref{1.21} to prove \eqref{1.18}.
We refer to the proof of \eqref{4.25} for details.
\par In the last
step, we show the low-bounds on the convergence rates of solutions.
First, we deal with $(n^+, n^-, u^+,u^-, \beta^+n^++\beta^-n^-)$. To
begin with, we will employ Plancherel theorem and careful analysis
on the solution semigroup to derive the lower convergence rate in
$\dot H^{-1}(\mathbb{R}^3)$ and $L^2(\mathbb{R}^3)$. Then, we can
prove the lower bound on the convergence rates in $L^2$-norm for the
higher-order spatial derivatives by an interpolation trick. Second,
we tackle with $(\rho^\pm-\bar{\rho}^\pm)$. By fully using the
expression of the pressure differential $dP$ and embedding estimate
of Riesz potential, we can get the lower convergence rate in
$L^2(\mathbb{R}^3)$. For $1\leq k\leq \ell$, we can prove the lower
bound on the $L^2$ convergence rates for the $k$--th spatial
derivatives by an interpolation trick.

%%%%%%%%%%%%%%%%%%%%%%%%%%%%%%%%%%%%%%%%%%%%%%%
\section{\leftline {\bf{Spectral analysis and linear $L^2$ estimates.}}}
\setcounter{equation}{0}
%%%%%%%%%%%%%%%%%%%%%%%%%%%%%%%%%%%%%%%%%%%%%%%
\subsection{Reformulation} In this subsection, we first reformulate the system.
Setting   \[ n^{\pm}=R^{\pm}-1,
\]
the Cauchy problem \eqref{1.14}--\eqref{1.15} can be reformulated as
\begin{equation}\label{2.1}
\left\{\begin{array}{l}
\partial_{t} n^++\operatorname{div}u^+=F_1, \\
\partial_{t}u^{+}+\beta_1\nabla n^++\beta_2\nabla
n^--\nu^+_1\Delta u^+-\nu^+_2\nabla\operatorname{div} u^+-\sigma^+\nabla\Delta n^+=F_2, \\
\partial_{t} n^-+\operatorname{div}u^-=F_3, \\
\partial_{t}u^{-}+\beta_3\nabla n^++\beta_4\nabla
n^--\nu^-_1\Delta u^--\nu^-_2\nabla\operatorname{div} u^--\sigma^-\nabla\Delta n^-=F_4, \\
(n^+, u^+, n^-, u^-)(x,0)=(n^+_0, u^+_0, n^-_0, u^-_0)(x)\rightarrow
(0, \vec{0}, 0, \vec{0}),~~~\hbox{as}~~~|x|\rightarrow +\infty,\\
\end{array}\right.
\end{equation}
where $\nu_{1}^{\pm}=\frac{\mu^{\pm}}{\bar{\rho}^{\pm}}$,
$\nu_{2}^{\pm}=\frac{\mu^{\pm}+\lambda^{\pm}}{\bar{\rho}^{\pm}}>0$,
$\beta_{1}=\frac{\mathcal{C}^{2}(1,1)
\bar{\rho}^{-}}{\bar{\rho}^{+}}$,
$\beta_{2}=\beta_{3}=\mathcal{C}^{2}(1,1)$,
$\beta_{4}=\frac{\mathcal{C}^{2}(1,1)
\bar{\rho}^{+}}{\bar{\rho}^{-}}$ (which imply
$\beta_{1}\beta_{4}=\beta_{2}\beta_{3}=\beta_{2}^2=\beta_{3}^2$),
 and
the nonlinear terms are given by
\begin{align}
\label{2.2}F_{1}=&-\operatorname{div}\left(n^{+} u^{+}\right), \\
F_{2}^{i}=&-g_+\left(n^{+}, n^{-}\right) \partial_{i}
n^{+}-\bar{g}_{+}\left(n^{+}, n^{-}\right) \partial_{i} n^{-}
-\left(u^{+} \cdot \nabla\right) u_{i}^{+} \nonumber\\
&+\mu^{+} h_{+}\left(n^{+}, n^{-}\right) \partial_{j}
n^{+} \partial_{j} u_{i}^{+}+\mu^{+} k_{+}\left(n^{+}, n^{-}\right) \partial_{j} n^{-} \partial_{j} u_{i}^{+} \nonumber\\
\label{2.3}&+\mu^{+} h_{+}\left(n^{+}, n^{-}\right) \partial_{j}
n^{+} \partial_{i} u_{j}^{+}+\mu^{+} k_{+}\left(n^{+}, n^{-}\right)
\partial_{j} n^{-}
 \partial_{i} u_{j}^{+}\\
&+\lambda^{+} h_{+}\left(n^{+}, n^{-}\right) \partial_{i} n^{+}
\partial_{j} u_{j}^{+}+\lambda^{+} k_{+}\left(n^{+},
n^{-}\right) \partial_{i} n^{-} \partial_{j} u_{j}^{+} \nonumber\\
&+\mu^{+} l_{+}\left(n^{+}, n^{-}\right) \partial_{j}^{2}
u_{i}^{+}+\left(\mu^{+}+\lambda^{+}\right) l_{+}\left(n^{+},
n^{-}\right) \partial_{i}
 \partial_{j} u_{j}^{+}, \nonumber\\
\label{2.4}F_{3}=&-\operatorname{div}\left(n^{-} u^{-}\right), \\
F_{4}^{i}=&-g_-\left(n^{+}, n^{-}\right) \partial_{i} n^{-}-
\bar{g}_{-}\left(n^{+}, n^{-}\right) \partial_{i} n^{+}-\left(u^{-} \cdot \nabla\right) u_{i}^{-}\nonumber \\
&+\mu^{-} h_{-}\left(n^{+}, n^{-}\right) \partial_{j} n^{+}
\partial_{j} u_{i}^{-}+\mu^{-} k_{-}\left(n^{+}, n^{-}\right)
\partial_{j} n^{-} \partial_{j} u_{i}^{-} \nonumber\\
\label{2.5}&+\mu^{-} h_{-}\left(n^{+}, n^{-}\right) \partial_{j}
n^{+}
\partial_{i} u_{j}^{-}+\mu^{-} k_{-}\left(n^{+}, n^{-}\right)
 \partial_{j} n^{-} \partial_{i} u_{j}^{-} \\
&+\lambda^{-} h_{-}\left(n^{+}, n^{-}\right) \partial_{i} n^{+}
\partial_{j} u_{j}^{-}+\lambda^{-} k_{-}\left(n^{+}, n^{-}\right)
 \partial_{i} n^{-} \partial_{j} u_{j}^{-}\nonumber \\
&+\mu^{-} l_{-}\left(n^{+}, n^{-}\right) \partial_{j}^{2}
u_{i}^{-}+\left(\mu^{-}+\lambda^{-}\right) l_{-}\left(n^{+},
n^{-}\right) \partial_{i} \partial_{j} u_{j}^{-},\nonumber
\end{align}

where
\begin{equation}\label{2.6}
\left\{\begin{array}{l}
g_{+}\left(n^{+}, n^{-}\right)=\frac{\left(\mathcal{C}^{2} \rho^{-}\right)\left(n^{+}+1, n^{-}+1\right)}{\rho^{+}\left(n^{+}+1, n^{-}+1\right)}-\frac{\left(\mathcal{C}^{2} \rho^{-}\right)(1,1)}{\rho^{+}(1,1)}, \\
g_{-}\left(n^{+}, n^{-}\right)=\frac{\left(\mathcal{C}^{2}
\rho^{+}\right) \left(n^{+}+1,
n^{-}+1\right)}{\rho^{-}\left(n^{+}+1, n^{-}+1\right)}-
\frac{\left(\mathcal{C}^{2} \rho^{+}\right)(1,1)}{\rho^{-}(1,1)},
\end{array}\right.\end{equation}

\begin{equation}\label{2.7}
\left\{\begin{array}{l} \bar{g}_{+}\left(n^{+},
n^{-}\right)=\mathcal{C}^{2}\left(n^{+}+1,
n^{-}+1\right)-\mathcal{C}^{2}\left(1, 1\right)\\
\bar{g}_{-}\left(n^{+}, n^{-}\right)=\mathcal{C}^{2}\left(n^{+}+1, n^{-}+1\right)-\mathcal{C}^{2}(1,1),\\
\end{array}\right.
\end{equation}

\begin{equation}\label{2.8}
\left\{\begin{array}{l}
h_{+}\left(n^{+}, n^{-}\right)=\frac{\left(\mathcal{C}^{2}\alpha^{-}\right)\left(n^{+}+1, n^{-}+1\right)}{(n^++1)s_{-}^{2}\left(n^{+}+1, n^{-}+1\right)},\\
h_{-}\left(n^{+}, n^{-}\right)=-\frac{\left(\mathcal{C}^{2}
\right)\left(n^{+}+1,
n^{-}+1\right)}{(\rho^-s_{-}^{2})\left(n^{+}+1, n^{-}+1\right)},
\end{array}\right.
\end{equation}

\begin{equation}\label{2.9}
\left\{\begin{array}{l}
k_{+}\left(n^{+}, n^{-}\right)=-\frac{\mathcal{C}^{2}\left(n^{+}+1, n^{-}+1\right)}{(n^++1)(s_{+}^{2}\rho^+)
\left(n^{+}+1, n^{-}+1\right)},\\
k_{-}\left(n^{+}, n^{-}\right)=-\frac{\left(\alpha^+\mathcal{C}^{2}\right)\left(n^{+}+1, n^{-}+1\right)}{(n^-+1)
s_{+}^{2}\left(n^{+}+1, n^{-}+1\right)},\\
\end{array}\right.
\end{equation}

\begin{equation}\label{2.10}
l_{\pm}(n^+, n^-)=\frac{1}{\rho_{\pm}\left(n^{+}+1,
n^{-}+1\right)}-\frac{1}{\rho_{\pm}\left(1, 1\right)}.
\end{equation}
\par Define
$\tilde{U}=(\tilde{n}^+,\tilde{u}^+,\tilde{n}^-,\tilde{u}^-)^t$. In
terms of the semigroup theory for evolutionary equation, we will
investigate the following initial value problem for the
corresponding linear system of \eqref{2.1}:
\begin{equation}
\begin{cases}
\tilde{U}_t=\mathcal B\tilde{U},\\
\tilde{U}\big|_{t=0}={U}_0,
\end{cases}   \label{2.11}
\end{equation}
where the operator $\mathcal B$ is given by
\begin{equation}\nonumber\mathcal B=\begin{pmatrix}
0&-\text{div}&0&0\\
-\beta_1\nabla+\sigma^+ \nabla\Delta&\nu_{1}^{+} \Delta+\nu_{2}^{+} \nabla \otimes \nabla&-\beta_2\nabla&0\\
0&0&0&-\text{div}\\
-\beta_3\nabla&0&-\beta_4\nabla+\sigma^- \nabla\Delta&\nu_{1}^{-}
\Delta+\nu_{2}^{-} \nabla \otimes \nabla
\end{pmatrix}.\end{equation}
Applying Fourier transform to the system \eqref{2.11}, one has
\begin{equation}
\begin{cases}
\widehat {\widetilde{{U}}}_t=\mathcal A(\xi)\widehat {\widetilde{U}},\\
\widehat {\widetilde{U}}\big|_{t=0}=\widehat U_0=(n^+_0, u^+_0,
n^-_0, u^-_0),
\end{cases}   \label{2.12}
\end{equation}
where $\widehat
{\widetilde{U}}(\xi,t)=\mathfrak{F}(\widetilde{U}(x,t))$,
$\xi=(\xi^1,\xi^2,\xi^3)^t$ and $\mathcal A(\xi)$ is defined by
\begin{equation}\nonumber\mathcal A(\xi)=\begin{pmatrix}
0&-{i}\xi^t&0&0\\
-{i}\beta_1\xi-{i}\sigma^+|\xi|^2\xi&-\nu_1^+|\xi|^2{\rm I}_{3\times 3}-\nu_2^+\xi\otimes\xi&-{i}\beta_2\xi&0\\
0&0&0&- {i}\xi^t\\
- {i}\beta_3\xi&0&-
{i}\beta_4-{i}\sigma^-|\xi|^2\xi&-\nu_1^-|\xi|^2{\rm I}_{3\times
3}-\nu_2^-\xi\otimes\xi
\end{pmatrix}.\end{equation}

To derive the linear time--decay estimates,  by using a real method
as in \cite{Kobayashi1, Kobayashi2}, one need to make a detailed
analysis on the properties of the semigroup. Unfortunately, it seems
untractable, since the system \eqref{2.12} has eight equations and
the matrix $\mathcal A(\xi)$ can not be diagonalizable (see
\cite{Sideris} pp.807 for example). To overcome this difficult, we
take Hodge decomposition to system \eqref{2.11} such that it can be
decoupled into three systems. One has four equations, and the other
two are classic heat equations. This key observation allows us to
derive the optimal linear convergence rates.

To begin with, let $\varphi^{\pm}=\Lambda^{-1}{\rm
div}\tilde{u}^{\pm}$ be the ``compressible part" of the velocities
$\tilde{u}^{\pm}$, and denote $\phi^{\pm}=\Lambda^{-1}{\rm
curl}\tilde{u}^{\pm}$ (with $({\rm curl} z)_i^j
=\partial_{x_j}z^i-\partial_{x_i}z^j$) by the ``incompressible part"
of the velocities $\tilde{u}^{\pm}$. Then, we can rewrite the system
\eqref{2.11} as follows:
\begin{equation}\label{2.13}
\begin{cases}
\partial_t{\tilde{{n}}^+}+\Lambda{\varphi^+}=0,\\
\partial_t{\varphi^+}-\beta_1\Lambda{\tilde{{n}}^+}-\beta_2\Lambda{\tilde{{n}}^-}+\nu^+\Lambda^2{\varphi^+}-\sigma^+\Lambda^3{\tilde{{n}}^+}=0,\\
\partial_t{\tilde{{n}}^-}+\Lambda{\varphi^-}=0,\\
\partial_t{\varphi^-}-\beta_3\Lambda{\tilde{{n}}^+}-\beta_4\Lambda{\tilde{{n}}^-}+\nu^-\Lambda^2{\varphi^-}-\sigma^-\Lambda^3{\tilde{{n}}^-}=0,\\
(\tilde{n}^+, \varphi^+, \tilde{n}^-, \varphi^-)\big|_{t=0}=({n}^+_0, \Lambda^{-1}{\rm div}{u}^{+}_0, {n}^-_0, \Lambda^{-1}{\rm div}{u}^{-}_0)(x),\\
\end{cases}
\end{equation}
and
\begin{equation}\label{2.14}
\begin{cases}
\partial_t\phi^++\nu^+_1\Lambda^2\phi^+=0,\\
\partial_t\phi^-+\nu^-_1\Lambda^2\phi^-=0,\\
(\phi^+,\phi^-)\big|_{t=0}=(\Lambda^{-1}{\rm curl}{u}^{+}_0,
\Lambda^{-1}{\rm curl}{u}^{-}_0)(x),
\end{cases}
\end{equation}
where $\nu^{\pm}=\nu^{\pm}_1+\nu^{\pm}_2$.

\subsection{Spectral analysis and linear $L^2$ estimates}

 In view of the semigroup
theory, we may represent the IVP \eqref{2.13} for $\mathcal
U=(\tilde{n}^+, \varphi^+, \tilde{n}^-, \varphi^-)^t$ as
\begin{equation}
\begin{cases}
\mathcal U_t=\mathcal B_1\mathcal U,\\
\mathcal U\big|_{t=0}=\mathcal U_0,
\end{cases}   \label{2.15}
\end{equation}
where the operator $\mathcal B_1$ is defined by
\begin{equation}\nonumber\mathcal B_1=\begin{pmatrix}
0&-\Lambda&0&0\\
\beta_1\Lambda+\sigma^+\Lambda^3&-\nu^+\Lambda^2&\beta_2\Lambda&0\\
0&0&0&-\Lambda\\
\beta_3\Lambda&0&\beta_4\Lambda+\sigma^-\Lambda^3&-\nu^-\Lambda^2
\end{pmatrix}.\end{equation}
Taking Fourier transform to system \eqref{2.15}, we obtain
\begin{equation}
\begin{cases}
\widehat {\mathcal U}_t=\mathcal A_1(\xi)\widehat {\mathcal U},\\
\widehat {\mathcal U}\big|_{t=0}=\widehat {\mathcal U}_0,
\end{cases}   \label{2.16}
\end{equation}
where $\widehat {\mathcal U}(\xi,t)=\mathfrak{F}({\mathcal U}(x,t))$
and $\mathcal A_1(\xi)$ is given by
\begin{equation}\label{2.17}\mathcal A_1(\xi)=\begin{pmatrix}
0&-|\xi|&0&0\\
\beta_1|\xi|+\sigma^+|\xi|^3&-\nu^+|\xi|^2&\beta_2|\xi|&0\\
0&0&0&-|\xi|\\
\beta_3|\xi|&0&\beta_4|\xi|+\sigma^-|\xi|^3&-\nu^-|\xi|^2
\end{pmatrix}.\end{equation}
We compute the eigenvalues of matrix $\mathcal A_1(\xi)$ from the
determinant
\begin{equation}\begin{split}\nonumber&{\rm det}(\lambda{\rm I}-\mathcal A_1(\xi))\\
&=\lambda^4+(\nu^+|\xi|^2+\nu^-|\xi|^2)\lambda^3+(\beta_1|\xi|^2+\beta_4|\xi|^2+\sigma^+|\xi|^4+\sigma^-|\xi|^4
+\nu^+\nu^-|\xi|^4)\lambda^2\\&\quad+(\beta_1\nu^-|\xi|^4+\beta_4\nu^+|\xi|^4+\nu^+\sigma^-|\xi|^6+\nu^-\sigma^+|\xi|^6)\lambda\\
&\quad+\beta_1\sigma^-|\xi|^6+\beta_4\sigma^+|\xi|^6+\sigma^+\sigma^-|\xi|^8\\
&=0,
\end{split}\end{equation}
which implies that matrix $\mathcal A_1(\xi)$ possesses four
eigenvalues:
\begin{equation}\label{2.18}\begin{cases}\lambda_1=-\frac{\beta_1\nu^++\beta_4\nu^-}{2(\beta_1+\beta_4)}|\xi|^2+\mathrm{i}\sqrt{\beta_1+\beta_4}|\xi|+\mathcal O(|\xi|^3),\\
\lambda_2=-\frac{\beta_1\nu^++\beta_4\nu^-}{2(\beta_1+\beta_4)}|\xi|^2-\mathrm{i}\sqrt{\beta_1+\beta_4}|\xi|+\mathcal O(|\xi|^3),\\
\lambda_3=\frac{-(\beta_1\nu^-+\beta_4\nu^+)+\sqrt{(\beta_1\nu^-+\beta_4\nu^+)^2-4(\beta_1+\beta_4)(\beta_1\sigma^-+\beta_4\sigma^+)}}{2(\beta_1+\beta_4)}|\xi|^2+\mathcal O(|\xi|^4),\\
\lambda_4=\frac{-(\beta_1\nu^-+\beta_4\nu^+)-\sqrt{(\beta_1\nu^-+\beta_4\nu^+)^2-4(\beta_1+\beta_4)(\beta_1\sigma^-+\beta_4\sigma^+)}}{2(\beta_1+\beta_4)}|\xi|^2+\mathcal O(|\xi|^4),\\
\end{cases}\end{equation}
Set $\mathcal
R=\sqrt{(\beta_1\nu^-+\beta_4\nu^+)^2-4(\beta_1+\beta_4)(\beta_1\sigma^-+\beta_4\sigma^+)}$,
$\tilde \lambda_3=\frac{-(\beta_1\nu^-+\beta_4\nu^+)+\mathcal
R}{2(\beta_1+\beta_4)}$, $\tilde
\lambda_4=\frac{-(\beta_1\nu^-+\beta_4\nu^+)-\mathcal{R}}{2(\beta_1+\beta_4)}$.\par
In what follows, we will divide the issue into two cases and discuss
them respectively.

 {\color{blue}Case 1: $\lambda_3\neq \lambda_4$.}
In this case, we have $P^{-1}\mathcal A_1(\xi)P=\begin{pmatrix}
\lambda_1&0&0&0\\
0&\lambda_2&0&0\\
0&0&\lambda_3&0\\
0&0&0&\lambda_4
\end{pmatrix}.$
Next, we derive the expressions of the project operators $P_i$ for
$i=1,2,3,4$. By making tedious calculations, it is clear that the
semigroup $\text{e}^{t\mathcal A_1}$ can be expressed as
\begin{align}\begin{split}\label{2.19}
 \text{e}^{t\mathcal A_1(\xi)}&=\text{e}^{\lambda_1t}P\begin{pmatrix}
1&0&0&0\\
0&0&0&0\\
0&0&0&0\\
0&0&0&0
\end{pmatrix}P^{-1}+\text{e}^{\lambda_2t}P\begin{pmatrix}
0&0&0&0\\
0&1&0&0\\
0&0&0&0\\
0&0&0&0
\end{pmatrix}P^{-1}\\
&\quad+\text{e}^{\lambda_3t}P\begin{pmatrix}
0&0&0&0\\
0&0&0&0\\
0&0&1&0\\
0&0&0&0
\end{pmatrix}P^{-1}+\text{e}^{\lambda_4t}P\begin{pmatrix}
0&0&0&0\\
0&0&0&0\\
0&0&0&0\\
0&0&0&1
\end{pmatrix}P^{-1}\\
&:=\sum_{i=1}^4\text{e}^{\lambda_it}P_i(\xi), \end{split}\end{align}
where the projectors $P_i(\xi)$ can be computed as
\begin{equation}\begin{split}\label{2.20}P_1(\xi)=&P\begin{pmatrix}
1&0&0&0\\
0&0&0&0\\
0&0&0&0\\
0&0&0&0
\end{pmatrix}P^{-1}\\=&\frac{1}{(\lambda_2-\lambda_1)(\lambda_3-\lambda_1)(\lambda_4-\lambda_1)}P(\lambda_2I-J)(\lambda_3I-J)(\lambda_4I-J)P^{-1}
\\=&\frac{1}{(\lambda_2-\lambda_1)(\lambda_3-\lambda_1)(\lambda_4-\lambda_1)}P(\lambda_2I-J)P^{-1}P(\lambda_3I-J)P^{-1}P(\lambda_4I-J)P^{-1}
\\=&\frac{(\lambda_2I-\mathcal{A}_1(\xi))(\lambda_3I-\mathcal{A}_1(\xi))(\lambda_4I-\mathcal{A}_1(\xi))}{(\lambda_2-\lambda_1)(\lambda_3-\lambda_1)
(\lambda_4-\lambda_1)},
\end{split}\end{equation}
\begin{equation}\begin{split}\label{2.21}P_2(\xi)=&P\begin{pmatrix}
0&0&0&0\\
0&1&0&0\\
0&0&0&0\\
0&0&0&0
\end{pmatrix}P^{-1}\\=&\frac{1}{(\lambda_1-\lambda_2)(\lambda_3-\lambda_2)(\lambda_4-\lambda_2)}P(\lambda_1I-J)(\lambda_3I-J)
(\lambda_4I-J)P^{-1}
\\=&\frac{1}{(\lambda_1-\lambda_2)(\lambda_3-\lambda_2)(\lambda_4-\lambda_2)}P(\lambda_1I-J)P^{-1}P(\lambda_3I-J)P^{-1}P(\lambda_4I-J)P^{-1}
\\=&\frac{(\lambda_1I-\mathcal{A}_1(\xi))(\lambda_3I-\mathcal{A}_1(\xi))(\lambda_4I-\mathcal{A}_1(\xi))}
{(\lambda_1-\lambda_2)(\lambda_3-\lambda_2)(\lambda_4-\lambda_2)},
\end{split}\end{equation}
\begin{align}\label{2.22}P_3(\xi)=&P\begin{pmatrix}
0&0&0&0\\
0&0&0&0\\
0&0&1&0\\
0&0&0&0
\end{pmatrix}P^{-1}=\frac{1}{(\lambda_1-\lambda_3)(\lambda_2-\lambda_3)(\lambda_4-\lambda_3)}P(\lambda_1I-J)(\lambda_2I-J)
(\lambda_4I-J)P^{-1}\notag\\
=&\frac{1}{(\lambda_1-\lambda_3)(\lambda_2-\lambda_3)(\lambda_4-\lambda_3)}P(\lambda_1I-J)P^{-1}P(\lambda_2I-J)P^{-1}P(\lambda_4I-J)P^{-1}
\\=&\frac{(\lambda_1I-\mathcal{A}_1(\xi))(\lambda_2I-\mathcal{A}_1(\xi))(\lambda_4I-\mathcal{A}_1(\xi))}{(\lambda_1-\lambda_3)
(\lambda_2-\lambda_3)(\lambda_4-\lambda_3)},\notag
\end{align}
\begin{equation}\begin{split}\label{2.23}P_4(\xi)=&P\begin{pmatrix}
0&0&0&0\\
0&0&0&0\\
0&0&0&0\\
0&0&0&1
\end{pmatrix}P^{-1}\\=&\frac{1}{(\lambda_1-\lambda_4)(\lambda_2-\lambda_4)(\lambda_3-\lambda_4)}P(\lambda_1I-J)
(\lambda_2I-J)(\lambda_3I-J)P^{-1}
\\=&\frac{1}{(\lambda_1-\lambda_4)(\lambda_2-\lambda_4)(\lambda_3-\lambda_4)}P(\lambda_1I-J)P^{-1}P(\lambda_2I-J)P^{-1}P(\lambda_3I-J)P^{-1}
\\=&\frac{(\lambda_1I-\mathcal{A}_1(\xi))(\lambda_2I-\mathcal{A}_1(\xi))(\lambda_3I-\mathcal{A}_1(\xi))}
{(\lambda_1-\lambda_4)(\lambda_2-\lambda_4)(\lambda_3-\lambda_4)}.
\end{split}\end{equation}
Consequently, we can represent the solution of IVP \eqref{2.16} as
\begin{equation}
\widehat {\mathcal U}(\xi,t)=\text{e}^{t\mathcal A_1(\xi)}\widehat
{\mathcal U}_0(\xi)=\left(\sum_{i=1}^4
\text{e}^{r_it}P_i(\xi)\right)\widehat {\mathcal
U}_0(\xi),\label{2.24}
\end{equation}
where
\begin{equation}\label{2.25}P_1(\xi)=\begin{pmatrix}
\frac{\beta_1}{2(\beta_1+\beta_4)}&\frac{\beta_1}{2(\beta_1+\beta_4)^\frac{3}{2}}\mathrm{i}&\frac{\beta_2}{2(\beta_1+\beta_4)}
&\frac{\beta_2}{2(\beta_1+\beta_4)^\frac{3}{2}}\mathrm{i}\\
-\frac{\beta_1}{2(\beta_1+\beta_4)^\frac{1}{2}}\mathrm{i}&\frac{\beta_1}{2(\beta_1+\beta_4)}&-\frac{\beta_2}{2(\beta_1+\beta_4)^\frac{1}{2}}\mathrm{i}&
\frac{\beta_2}{2(\beta_1+\beta_4)}\\
\frac{\beta_2}{2(\beta_1+\beta_4)}&\frac{\beta_2}{2(\beta_1+\beta_4)^\frac{3}{2}}\mathrm{i}&\frac{\beta_4}{2(\beta_1+\beta_4)}
&\frac{\beta_4}{2(\beta_1+\beta_4)^\frac{3}{2}}\mathrm{i}\\
-\frac{\beta_2}{2(\beta_1+\beta_4)^\frac{1}{2}}\mathrm{i}&\frac{\beta_2}{2(\beta_1+\beta_4)}&-\frac{\beta_4}{2(\beta_1+\beta_4)^\frac{1}{2}}\mathrm{i}&
\frac{\beta_4}{2(\beta_1+\beta_4)}
\end{pmatrix}+\mathcal O(|\xi|),\end{equation}
\begin{equation}\label{2.26}P_2(\xi)=\begin{pmatrix}
\frac{\beta_1}{2(\beta_1+\beta_4)}&-\frac{\beta_1}{2(\beta_1+\beta_4)^\frac{3}{2}}\mathrm{i}&\frac{\beta_2}{2(\beta_1+\beta_4)}
&-\frac{\beta_2}{2(\beta_1+\beta_4)^\frac{3}{2}}\mathrm{i}\\
\frac{\beta_1}{2(\beta_1+\beta_4)^\frac{1}{2}}\mathrm{i}&\frac{\beta_1}{2(\beta_1+\beta_4)}&\frac{\beta_2}{2(\beta_1+\beta_4)^\frac{1}{2}}\mathrm{i}&
\frac{\beta_2}{2(\beta_1+\beta_4)}\\
\frac{\beta_2}{2(\beta_1+\beta_4)}&-\frac{\beta_2}{2(\beta_1+\beta_4)^\frac{3}{2}}\mathrm{i}&\frac{\beta_4}{2(\beta_1+\beta_4)}
&-\frac{\beta_4}{2(\beta_1+\beta_4)^\frac{3}{2}}\mathrm{i}\\
\frac{\beta_2}{2(\beta_1+\beta_4)^\frac{1}{2}}\mathrm{i}&\frac{\beta_2}{2(\beta_1+\beta_4)}&\frac{\beta_4}{2(\beta_1+\beta_4)^\frac{1}{2}}\mathrm{i}&
\frac{\beta_4}{2(\beta_1+\beta_4)}
\end{pmatrix}+\mathcal O(|\xi|),\end{equation}
\begin{equation}\begin{split}\label{2.27}P_3(\xi)=&\begin{pmatrix}
\frac{\beta_1\beta_4(\nu^+-\nu^-)}{(\beta_1+\beta_4)\mathcal
R}-\frac{\beta_4\tilde r_4}{\mathcal R}&-\frac{\beta_4}{\mathcal
R|\xi|}&\frac{\beta_2\beta_4(\nu^+-\nu^-)}{(\beta_1+\beta_4)\mathcal
R}+\frac{\beta_2\tilde \lambda_4}{\mathcal R}
&\frac{\beta_2}{\mathcal R|\xi|}\\
0&-\frac{\beta_4(\beta_4\nu^++\beta_1\nu^-)}{(\beta_1+\beta_4)\mathcal
R}-\frac{\beta_4\tilde \lambda_4}{\mathcal R}&0&
\frac{\beta_2(\beta_4\nu^++\beta_1\nu^-)}{(\beta_1+\beta_4)\mathcal R}+\frac{\beta_2\tilde r_4}{\mathcal R}\\
\frac{\beta_1\beta_2(\nu^--\nu^+)}{(\beta_1+\beta_4)\mathcal
R}+\frac{\beta_2\tilde \lambda_4}{\mathcal
R}&\frac{\beta_2}{\mathcal
R|\xi|}&\frac{\beta_1\beta_4(\nu^--\nu^+)}{(\beta_1+\beta_4)\mathcal
R}-\frac{\beta_1\tilde \lambda_4}{\mathcal R}
&-\frac{\beta_1}{\mathcal R|\xi|}\\
0&\frac{\beta_2(\beta_4\nu^++\beta_1\nu^-)}{(\beta_1+\beta_4)\mathcal
R}+\frac{\beta_2\tilde \lambda_4}{\mathcal R}&0&
-\frac{\beta_1(\beta_4\nu^++\beta_1\nu^-)}{(\beta_1+\beta_4)\mathcal
R}-\frac{\beta_1\tilde \lambda_4}{\mathcal R}
\end{pmatrix}\\
&+\mathcal O(|\xi|),\end{split}\end{equation}
\begin{equation}\begin{split}\label{2.28}P_4(\xi)=&-\begin{pmatrix}
\frac{\beta_1\beta_4(\nu^+-\nu^-)}{(\beta_1+\beta_4)\mathcal
R}-\frac{\beta_4\tilde \lambda_3}{\mathcal
R}&-\frac{\beta_4}{\mathcal
R|\xi|}&\frac{\beta_2\beta_4(\nu^+-\nu^-)}{(\beta_1+\beta_4)\mathcal
R}+\frac{\beta_2\tilde \lambda_3}{\mathcal R}
&\frac{\beta_2}{\mathcal R|\xi|}\\
0&-\frac{\beta_4(\beta_4\nu^++\beta_1\nu^-)}{(\beta_1+\beta_4)\mathcal
R}-\frac{\beta_4\tilde \lambda_3}{\mathcal R}&0&
\frac{\beta_2(\beta_4\nu^++\beta_1\nu^-)}{(\beta_1+\beta_4)\mathcal R}+\frac{\beta_2\tilde \lambda_3}{\mathcal R}\\
\frac{\beta_1\beta_2(\nu^--\nu^+)}{(\beta_1+\beta_4)\mathcal
R}+\frac{\beta_2\tilde \lambda_3}{\mathcal
R}&\frac{\beta_2}{\mathcal
R|\xi|}&\frac{\beta_1\beta_4(\nu^--\nu^+)}{(\beta_1+\beta_4)\mathcal
R}-\frac{\beta_1\tilde \lambda_3}{\mathcal R}
&-\frac{\beta_1}{\mathcal R|\xi|}\\
0&\frac{\beta_2(\beta_4\nu^++\beta_1\nu^-)}{(\beta_1+\beta_4)\mathcal
R}+\frac{\beta_2\tilde \lambda_3}{\mathcal R}&0&
-\frac{\beta_1(\beta_4\nu^++\beta_1\nu^-)}{(\beta_1+\beta_4)\mathcal
R}-\frac{\beta_1\tilde \lambda_3}{\mathcal R}
\end{pmatrix}\\
&+\mathcal O(|\xi|).\end{split}\end{equation}

{\color{blue}Case 2: $\lambda_3=\lambda_4$.} In this case,
 we have $P^{-1}\mathcal A_1P=\begin{pmatrix}
\lambda_1&0&0&0\\
0&\lambda_2&0&0\\
0&0&\lambda_3&1\\
0&0&0&\lambda_3
\end{pmatrix}=J.$
Compared to {\color{blue}Case 1}, where the matrix $\mathcal
A_1(\xi)$ can be diagonalizable, it is rather difficult to get the
expression of the solution in terms of the solution semigroup as in
\eqref{2.24} due to the appearance of the multiple roots. The key
idea here is that we first conjecture
 that a similar formula as
\eqref{2.24} holds, and then derive the explicit expressions of
$P_i(\xi)$ for $i=1,2, 3, 4.$ To see this, we first notice that
\begin{equation}\nonumber
\begin{cases}
\widehat {\mathcal U}_t(\xi, t)=PJP^{-1}\widehat {\mathcal U}(\xi, t),\\
\widehat {\mathcal U}(\xi, 0)=\widehat {\mathcal U}_0(\xi),
\end{cases}
\end{equation}
which implies
\begin{equation}\label{2.29}\widehat {\mathcal U}(\xi, t)=P\begin{pmatrix}
\text{e}^{\lambda_1t}&0&0&0\\
0&\text{e}^{\lambda_2t}&0&0\\
0&0&\text{e}^{\lambda_3t}&t\text{e}^{\lambda_3t}\\
0&0&0&\text{e}^{\lambda_3t}
\end{pmatrix}P^{-1}\widehat {\mathcal U}_0(\xi).
\end{equation}
Inspired by {\color{blue} Case 1}, we obtain that the following
decomposition holds:
\begin{equation}\begin{split}\label{2.30}\widehat {\mathcal U}(\xi, t)=&P\begin{pmatrix}
\text{e}^{\lambda_1t}&0&0&0\\
0&\text{e}^{\lambda_2t}&0&0\\
0&0&\text{e}^{\lambda_3t}&t\text{e}^{\lambda_3t}\\
0&0&0&\text{e}^{\lambda_3t}
\end{pmatrix}P^{-1}\widehat {\mathcal
U}_0(\xi)\\=&\left(\text{e}^{\lambda_1t}P_1(\xi)+\text{e}^{\lambda_2t}P_2(\xi)+\text{e}^{\lambda_3t}P_3(\xi)+t\text{e}^{\lambda_3t}
P_4(\xi)\right)\widehat {\mathcal U}_0(\xi).
\end{split}\end{equation}
Then, after careful and sophisticated calculations, we find the
explicit expressions of $P_i(\xi)$ for $i=1,2, 3, 4,$  which are
given by
\begin{equation}\begin{split}\label{2.31}P_1(\xi)=&P\begin{pmatrix}
1&0&0&0\\
0&0&0&0\\
0&0&0&0\\
0&0&0&0
\end{pmatrix}P^{-1}\\=&\frac{1}{(\lambda_2-\lambda_1)(\lambda_3-\lambda_1)^2}P(\lambda_2I-J)(\lambda_3I-J)^2P^{-1}
\\=&\frac{1}{(\lambda_2-\lambda_1)(\lambda_3-\lambda_1)^2}P(\lambda_2I-J)P^{-1}P(\lambda_3I-J)P^{-1}P(\lambda_3I-J)P^{-1}
\\=&\frac{(r_2I-A)(r_3I-A)^2}{(r_2-r_1)(r_3-r_1)^2}\\=&\begin{pmatrix}
\frac{\beta_1}{2(\beta_1+\beta_4)}&\frac{\beta_1}{2(\beta_1+\beta_4)^\frac{3}{2}}\mathrm{i}&\frac{\beta_2}{2(\beta_1+\beta_4)}
&\frac{\beta_2}{2(\beta_1+\beta_4)^\frac{3}{2}}\mathrm{i}\\
-\frac{\beta_1}{2(\beta_1+\beta_4)^\frac{1}{2}}\mathrm{i}&\frac{\beta_1}{2(\beta_1+\beta_4)}&-\frac{\beta_2}{2(\beta_1+\beta_4)^\frac{1}{2}}\mathrm{i}&
\frac{\beta_2}{2(\beta_1+\beta_4)}\\
\frac{\beta_2}{2(\beta_1+\beta_4)}&\frac{\beta_2}{2(\beta_1+\beta_4)^\frac{3}{2}}\mathrm{i}&\frac{\beta_4}{2(\beta_1+\beta_4)}
&\frac{\beta_4}{2(\beta_1+\beta_4)^\frac{3}{2}}\mathrm{i}\\
-\frac{\beta_2}{2(\beta_1+\beta_4)^\frac{1}{2}}\mathrm{i}&\frac{\beta_2}{2(\beta_1+\beta_4)}&-\frac{\beta_4}{2(\beta_1+\beta_4)^\frac{1}{2}}\mathrm{i}&
\frac{\beta_4}{2(\beta_1+\beta_4)}
\end{pmatrix}+\mathcal O(|\xi|),
\end{split}\end{equation}

\begin{equation}\begin{split}\label{2.32}
P_2(\xi)=&P\begin{pmatrix}
0&0&0&0\\
0&1&0&0\\
0&0&0&0\\
0&0&0&0
\end{pmatrix}P^{-1}=\frac{1}{(\lambda_1-\lambda_2)(\lambda_3-\lambda_2)^2}P(\lambda_1I-J)(\lambda_3I-J)^2P^{-1}
\\=&\frac{1}{(\lambda_1-\lambda_2)(\lambda_3-\lambda_2)^2}P(\lambda_1I-J)P^{-1}P(\lambda_3I-J)P^{-1}
P(\lambda_3I-J)P^{-1}=\frac{(\lambda_1I-A)(\lambda_3I-A)^2}{(\lambda_1-\lambda_2)(\lambda_3-\lambda_2)^2}\\=&\begin{pmatrix}
\frac{\beta_1}{2(\beta_1+\beta_4)}&-\frac{\beta_1}{2(\beta_1+\beta_4)^\frac{3}{2}}\mathrm{i}&\frac{\beta_2}{2(\beta_1+\beta_4)}
&-\frac{\beta_2}{2(\beta_1+\beta_4)^\frac{3}{2}}\mathrm{i}\\
\frac{\beta_1}{2(\beta_1+\beta_4)^\frac{1}{2}}\mathrm{i}&\frac{\beta_1}{2(\beta_1+\beta_4)}&\frac{\beta_2}{2(\beta_1+\beta_4)^\frac{1}{2}}\mathrm{i}&
\frac{\beta_2}{2(\beta_1+\beta_4)}\\
\frac{\beta_2}{2(\beta_1+\beta_4)}&-\frac{\beta_2}{2(\beta_1+\beta_4)^\frac{3}{2}}\mathrm{i}&\frac{\beta_4}{2(\beta_1+\beta_4)}
&-\frac{\beta_4}{2(\beta_1+\beta_4)^\frac{3}{2}}\mathrm{i}\\
\frac{\beta_2}{2(\beta_1+\beta_4)^\frac{1}{2}}\mathrm{i}&\frac{\beta_2}{2(\beta_1+\beta_4)}&\frac{\beta_4}{2(\beta_1+\beta_4)^\frac{1}{2}}\mathrm{i}&
\frac{\beta_4}{2(\beta_1+\beta_4)}
\end{pmatrix}+\mathcal O(|\xi|),
\end{split}\end{equation}
\begin{equation}\begin{split}\label{2.33}P_4(\xi)=&P\begin{pmatrix}
0&0&0&0\\
0&0&0&0\\
0&0&0&1\\
0&0&0&0
\end{pmatrix}P^{-1}=-\frac{1}{(\lambda_1-\lambda_3)(\lambda_2-\lambda_3)}P(\lambda_1I-J)(\lambda_2I-J)(\lambda_3I-J)P^{-1}
\\=&
-\frac{(\lambda_1I-A)(\lambda_2I-A)(\lambda_3I-A)}{(\lambda_1-\lambda_3)(\lambda_2-\lambda_3)}\\=&\begin{pmatrix}
\frac{\beta_1\beta_4(\nu^+-\nu^-)}{2(\beta_1+\beta_4)^2}+\frac{\beta_4\nu^+}{2(\beta_1+\beta_4)}&-\frac{\beta_4}{(\beta_1+\beta_4)|\xi|}&\frac{\beta_2\beta_4(\nu^+-\nu^-)}{2(\beta_1+\beta_4)^2}
-\frac{\beta_2\nu^-}{2(\beta_1+\beta_4)}
&\frac{\beta_2}{(\beta_1+\beta_4)|\xi|}\\
0&-\frac{\beta_4(\beta_4\nu^++\beta_1\nu^-)}{2(\beta_1+\beta_4)^2}&0&
\frac{\beta_2(\beta_4\nu^++\beta_1\nu^-)}{2(\beta_1+\beta_4)^2}\\
\frac{\beta_1\beta_2(\nu^--\nu^+)}{2(\beta_1+\beta_4)^2}-\frac{\beta_2\nu^+}{2(\beta_1+\beta_4)}&\frac{\beta_2}{(\beta_1+\beta_4)|\xi|}
&\frac{\beta_1\beta_4(\nu^--\nu^+)}{2(\beta_1+\beta_4)^2}+\frac{\beta_1\nu^-}{2(\beta_1+\beta_4)}
&-\frac{\beta_1}{(\beta_1+\beta_4)|\xi|}\\
0&\frac{\beta_2(\beta_4\nu^++\beta_1\nu^-)}{2(\beta_1+\beta_4)^2}&0&
-\frac{\beta_1(\beta_4\nu^++\beta_1\nu^-)}{2(\beta_1+\beta_4)^2}
\end{pmatrix}|\xi|^2\\
&\quad+\mathcal O(|\xi|^3),
\end{split}\end{equation}

\begin{align}\label{2.34}P_3(\xi)=&P\begin{pmatrix}
0&0&0&0\\
0&0&0&0\\
0&0&1&0\\
0&0&0&1
\end{pmatrix}P^{-1}\notag\\=&\frac{1}{(\lambda_1-\lambda_3)(\lambda_2-\lambda_3)}P\left[(\lambda_1I-J)(\lambda_2I-J)+(\lambda_1+\lambda_2-2\lambda_3)
\begin{pmatrix}
0&0&0&0\\
0&0&0&0\\
0&0&0&1\\
0&0&0&0
\end{pmatrix}\right]P^{-1}\notag
\\=&\frac{(\lambda_1I-A)(\lambda_2I-A)}{(\lambda_1-\lambda_3)(\lambda_2-\lambda_3)}
+\frac{\lambda_1+\lambda_2-2\lambda_3}{(\lambda_1-\lambda_3)(\lambda_2-\lambda_3)}P_4(\xi)\notag\\=&\begin{pmatrix}
\frac{\beta_4}{\beta_1+\beta_4}&0&-\frac{\beta_2}{\beta_1+\beta_4}
&0\\
0&\frac{\beta_4}{\beta_1+\beta_4}&0&
-\frac{\beta_2}{\beta_1+\beta_4}\\
-\frac{\beta_2}{\beta_1+\beta_4}&0 &\frac{\beta_1}{\beta_1+\beta_4}
&0\\
0&-\frac{\beta_2}{\beta_1+\beta_4}&0&
\frac{\beta_1}{\beta_1+\beta_4}
\end{pmatrix}+\mathcal O(|\xi|).
\end{align}

By virtue of \eqref{2.24}-\eqref{2.28} and
\eqref{2.30}-\eqref{2.34}, we can establish the following estimates
on low-frequency part of the solution $\widehat {\mathcal U}(\xi,
t)$ to the IVP \eqref{2.16}.

\begin{Lemma}\label{Lemma2.1} Let $\bar{\nu}=\min\Big\{\frac{\beta_1\nu^++\beta_4\nu^-}{2(\beta_1+\beta_4)},-\tilde{\lambda}_3\Big\}$
if $\mathcal R$ is a real number, and
$\bar{\nu}=\min\Big\{\frac{\beta_1\nu^++\beta_4\nu^-}{2(\beta_1+\beta_4)},\frac{\beta_1\nu^-+\beta_4\nu^+}{2(\beta_1+\beta_4)}\Big\}$
if $\mathcal R$ is an imaginary number, then there exists a
sufficiently small positive constant $\eta$, such that the following
estimates hold
\begin{equation}\label{2.35}
\displaystyle \left|\widehat{\tilde{n}^+}\right|,
\left|\widehat{\tilde{n}^-}\right|\lesssim
\frac{\text{e}^{-\bar{\nu}|\xi|^2t}}{|\xi|}\left(
\left|\widehat{n^+_0}\right|
 +\left|\widehat{\varphi^+_0}\right|+ \left|\widehat{n^-_0}\right|+
 \left|\widehat{\varphi^-_0}\right|\right)+t\text{e}^{-\bar{\nu}|\xi|^2t}\left(\left|\widehat{n^+_0}\right|
 +\left|\widehat{\varphi^+_0}\right|+ \left|\widehat{n^-_0}\right|+\left|\widehat{\varphi^-_0}\right|\right)
\end{equation}
\begin{equation}\label{2.36}
 \displaystyle\left|\widehat{\varphi^+}\right|, \left|\widehat{\varphi^-}\right|, \left|\beta_1
 \widehat{\tilde{n}^+}+\beta_2\widehat{\tilde{n}^-}\right|\lesssim \text{e}^{-\bar{\nu}|\xi|^2t}
 \left(\left|\widehat{n^+_0}\right|+\left|\widehat{\varphi^+_0}\right|+\left|\widehat{n^-_0}\right|+
 \left|\widehat{\varphi^-_0}\right|\right),
\end{equation}
for any $|\xi|\le \eta$.
\end{Lemma}
The key estimates in \eqref{2.35} and \eqref{2.36} enable us to
establish the optimal $L^2$-convergence rate on the low-frequency
part of the solution, which is stated in the following proposition.
\begin{Proposition}[$L^2$--theory]\label{Proposition2.2} For any $k>-\frac{3}{2}$, there exists a positive
constant $C$ independent of $t$, such that
\begin{equation}\left\|\nabla ^k \left(\tilde{n}^{+,l},\tilde{n}^{-,l}\right)(t)\right\|_{L^2}\leq
C(1+t)^{-\frac{1}{4}-\frac{k}{2}}\left\|\widehat {\mathcal
U}^l(0)\right\|_{L^\infty},\label{2.37}\end{equation}\vspace{-0.2cm}
\begin{equation}\left\|\nabla^k \left(\varphi^{+,l},\varphi^{-,l},\beta_1\tilde{n}^{+,l}+\beta_2\tilde{n}^{-,l}\right)(t)\right\|_{L^2}\leq
C(1+t)^{-\frac{3}{4}-\frac{k}{2}}\left\|\widehat {\mathcal
U}^l(0)\right\|_{L^\infty},\label{2.38}\end{equation}
 for any $t\geq 0$.
\end{Proposition}
It should be mentioned that the above $L^2$-convergence rates are
optimal. As a matter of fact, we also have the lower bounds on the
convergence rates which are stated in the following proposition.
\begin{Proposition}\label{Proposition2.3} Assume that $\left(n^{+,l}_0,\varphi^{+,l}_0,n^{-,l}_0,\varphi^{-,l}_0\right)\in L^1$
satisfies
\begin{equation}\widehat{n^{+,l}_0}(\xi)=\widehat{\varphi^{+,l}_0}(\xi)=\widehat{n^{-,l}_0}(\xi)=0,
\quad \text{and} \quad \widehat{\varphi^{-,l}_0}(\xi)-c_0\sim
|\xi|^s\label{2.39}\end{equation} for any $|\xi|\le\eta$, where
$c_0$ and $s$ are given two positive constants. Then, there exists a
positive constant $C_1$ independent of $t$, such that it holds that
\begin{equation}\label{2.40}\min\left\{\| \tilde{n}^{+,l}(t)\|_{L^2},\|\tilde{n}^{-,l}(t)\|_{L^2}\right\}\ge
C_1c_0(1+t)^{-\frac{1}{4}},\end{equation}
\begin{equation}\label{2.41}\min\left\{\|\varphi^{+,l}(t)\|_{L^2},\|\varphi^{-,l}(t)\|_{L^2},\|(\beta_1\tilde{n}^{+,l}+\beta_2\tilde{n}^{-,l}
)(t)\|_{L^2}\right\} \ge C_1c_0(1+t)^{-\frac{3}{4}},\end{equation}
for sufficiently large $t$.
\end{Proposition}
\begin{proof} We only deal with { \color{blue} Case 1}, since  the proof of { \color{blue} Case 2} is similar.
Because the proofs are similar, we will  focus on the proofs
concerning the terms $\varphi^{+,l}$ and $\tilde{n}^{+,l}$ for
simplicity. To begin with, by virtue of \eqref{2.24} and
\eqref{2.30}, we have
\begin{align*}
\widehat{\varphi^{+,l}}=&~\left(\frac{\beta_2}{2(\beta_1+\beta_4)}
+\mathcal O(|\xi|)\right)\text{e}^{\lambda_1(|\xi|)t}\widehat{\varphi_0^{-,l}}
+\left(\frac{\beta_2}{2(\beta_1+\beta_4)}+\mathcal O(|\xi|)\right)\text{e}^{\lambda_2(|\xi|)t}\widehat{\varphi_0^{-,l}}\\
&+\left(\frac{\beta_2(\beta_4\nu^++\beta_1\nu^-)}{(\beta_1+\beta_4)\mathcal
R}+\frac{\beta_2\tilde r_4}{\mathcal R}+\mathcal
O(|\xi|)\right)\text{e}^{\lambda_3(|\xi|)t}\widehat{\varphi_0^{-,l}}
\\&+\left(\frac{\beta_2(\beta_4\nu^++\beta_1\nu^-)}{(\beta_1+\beta_4)\mathcal
R}+\frac{\beta_2\tilde \lambda_3}{\mathcal R}+\mathcal
O(|\xi|)\right)\text{e}^{\lambda_4(|\xi|)t}\widehat{\varphi_0^{-,l}}
\\ \sim &~\frac{\beta_2}{\beta_1+\beta_4}\text{e}^{-\frac{\beta_1\nu^++\beta_4\nu^-}{2(\beta_1+\beta_4)}|\xi|^2t}
\cos\left[(\sqrt{\beta_1+\beta_4}|\xi|+O(|\xi|^3))t\right]\widehat{\varphi_0^{-,l}}\\
&+\left(\frac{\beta_2(\beta_4\nu^++\beta_1\nu^-)}{(\beta_1+\beta_4)\mathcal
R}+\frac{\beta_2\tilde \lambda_4}{\mathcal R}\right)\text{e}^{\tilde
\lambda_3|\xi|^2t}\widehat{\varphi_0^{-,l}}+\left(\frac{\beta_2(\beta_4\nu^++\beta_1\nu^-)}{(\beta_1+\beta_4)\mathcal
R}+\frac{\beta_2\tilde \lambda_3}{\mathcal R}\right)\text{e}^{\tilde
\lambda_4|\xi|^2t}\widehat{\varphi_0^{-,l}},
\end{align*}
which together with Plancherel theorem and the double angle formula
implies
\begin{align}\label{2.42}\|\varphi^{+,l}\|_{L^2}^2=&~\|\widehat{\varphi^{+,l}}\|_{L^2}^2\nonumber\\
\ge
&~\frac{c_0^2\beta_2^2}{2(\beta_1+\beta_4)^2}\int_{|\xi|\le\frac{\eta}{2}}\text{e}^{-\frac{\beta_1\nu^+
+\beta_4\nu^-}{\beta_1+\beta_4}|\xi|^2t}\cos^2\left[(\sqrt{\beta_1+\beta_4}|\xi|
+O(|\xi|^3))t\right]\mathrm{d}\xi \nonumber\\
&+\frac{2\beta_2c_0^2}{\beta_1+\beta_4}\int_{|\xi|\le
\eta}\text{e}^{-\frac{\beta_1\nu^++\beta_4\nu^-}{2(\beta_1+\beta_4)}
|\xi|^2t}\cos\left[(\sqrt{\beta_1+\beta_4}|\xi|
+O(|\xi|^3))t\right]\nonumber\\&\times\left(\frac{\beta_2(\beta_4\nu^++\beta_1\nu^-)}{(\beta_1+\beta_4)\mathcal
R}+\frac{\beta_2\tilde \lambda_4}{\mathcal R}\right)\text{e}^{\tilde
\lambda_3|\xi|^2t}\widehat{\varphi_0^{-,l}}+\left(\frac{\beta_2(\beta_4\nu^++\beta_1\nu^-)}{(\beta_1+\beta_4)\mathcal
R}+\frac{\beta_2\tilde \lambda_3}{\mathcal R}\right)\text{e}^{\tilde
\lambda_4|\xi|^2t}\widehat{\varphi_0^{-,l}}\mathrm{d}\xi\nonumber\\&-C_1\int_{|\xi|\le
\eta}\text{e}^{-\frac{\beta_1\nu^++\beta_4\nu^-}{2(\beta_1+\beta_4)}|\xi|^2t}|\xi|^{2s}\mathrm{d}\xi
\nonumber\\
\ge &~\frac{c_0^2\beta_2^2}{4(\beta_1+\beta_4)^2}\int_
{|\xi|\le\frac{\eta}{2}}\text{e}^{-\frac{\beta_1\nu^++\beta_4\nu^-}{\beta_1+\beta_4}|\xi|^2t}\mathrm{d}\xi
\\&+\frac{c_0^2\beta_2^2}{4(\beta_1+\beta_4)^2}
\int_{|\xi|\le\frac{\eta}{2}}\text{e}^{-\frac{\beta_1\nu^++\beta_4\nu^-}{\beta_1+\beta_4}|\xi|^2t}\cos\left[2(\sqrt{\beta_1+\beta_4}|\xi|
+O(|\xi|^3))t\right]\mathrm{d}\xi \nonumber\\
&+\frac{2\beta_2c_0^2}{\beta_1+\beta_4}\int_ {|\xi|\le
\eta}\text{e}^{-\frac{\beta_1\nu^++\beta_4\nu^-}{2(\beta_1+\beta_4)}|\xi|^2t}\cos\left[(\sqrt{\beta_1+\beta_4}|\xi|
+O(|\xi|^3))t\right]\nonumber\\&\times\left(\frac{\beta_2(\beta_4\nu^++\beta_1\nu^-)}{(\beta_1+\beta_4)\mathcal
R}+\frac{\beta_2\tilde \lambda_4}{\mathcal R}\right)\text{e}^{\tilde
\lambda_3|\xi|^2t}\widehat{\varphi_0^{-,l}}+\left(\frac{\beta_2(\beta_4\nu^++\beta_1\nu^-)}{(\beta_1+\beta_4)\mathcal
R}+\frac{\beta_2\tilde \lambda_3}{\mathcal R}\right)\text{e}^{\tilde
\lambda_4|\xi|^2t}\widehat{\varphi_0^{-,l}}\nonumber\\&-C_1\int_{|\xi|\le
\eta}\text{e}^{-\frac{\beta_1\nu^++\beta_4\nu^-}{2(\beta_1+\beta_4)}|\xi|^2t}|\xi|^{2s}\mathrm{d}\xi
\nonumber\\
\ge &~ C_2c_0(1+t)^{-\frac{3}{2}}-C_3\left((1+t)^{-\frac{3}{2}-\frac{1}{2}}+(1+t)^{-\frac{3}{2}-2s}\right)\nonumber\\
\ge &~ C_4c_0(1+t)^{-\frac{3}{2}}\nonumber\end{align}
 if $t$ large enough. Similarly, for the term ${\tilde{n}^{+,l}}$, we have
\begin{align*}\widehat{\tilde{n}^{+,l}}=&~\left(\frac{\beta_2}{2(\beta_1+\beta_4)^\frac{3}{2}}
\mathrm{i}+\mathcal
O(|\xi|)\right)\left(\text{e}^{\lambda_1(|\xi|)t}-\text{e}^{\lambda_2(|\xi|)t}\right)
\widehat{\varphi_0^{-,l}}\\&+\left(\frac{\beta_2}{\mathcal R|\xi|}
+\mathcal O(|\xi|)\right)
\left(\text{e}^{\lambda_3(|\xi|)t}-\text{e}^{\lambda_4(|\xi|)t}\right)\widehat{\varphi_0^{-,l}}\\
\sim&\left(\frac{\beta_2}{2(\beta_1+\beta_4)^\frac{3}{2}}\mathrm{i}+\mathcal
O(|\xi|)\right)\left(\text{e}^{\lambda_1(|\xi|)t}-\text{e}^{\lambda_2(|\xi|)t}\right)\widehat{\varphi_0^{-,l}}
\\&+\left(\frac{\beta_2}{\mathcal
R|\xi|} +\mathcal O(|\xi|)\right){\mathcal R|\xi|}\frac{\mathcal
R|\xi|^2t}{\beta_1+\beta_4}
\text{e}^{\lambda_3(|\xi|)t}\widehat{\varphi_0^{-,l}},
\end{align*}
which together with Plancherel theorem leads to
\begin{equation}\label{2.43}\begin{split}\left\|\tilde{n}^{+,l}\right\|_{L^2}^2
=&~\left\|\widehat{\tilde{n}^{+,l}}\right\|_{L^2}^2\ge
C_5c_0(1+t)^{-\frac{1}{2}}.\end{split}\end{equation}\par Therefore,
we have completed the proof of Proposition \ref{Proposition2.3}.
\end{proof}
\smallskip
In terms of the classic theory of the heat equation, the solutions
$\phi^\pm$ to the IVP \eqref{2.14} possess the following decay
estimates.
\begin{Proposition}[$L^2$--theory]\label{Proposition2.4} For any
$k>-\frac{3}{2}$, there exists a positive constant $C$ independent
of $t$ such that
\begin{equation}\left\|\nabla ^k\phi^{\pm,l}(t)\right\|_{L^2}\lesssim(1+t)^{-\frac{3}{4}-\frac{k}{2}}
\left\| {\widehat{\phi^{\pm,
l}}}(0)\right\|_{L^{\infty}},\label{2.44}
\end{equation}
 for any $t\geq 0$.
\end{Proposition}

\smallskip

Finally, noticing  the definition of $\varphi^{\pm}$ and
$\phi^{\pm}$, and the fact that the relations
$$\tilde{u}^{\pm}=-\wedge^{-1}\nabla\varphi^{\pm}-\wedge^{-1}\text{div}\phi^{\pm}$$
involves pseudo-differential operators of degree zero, the estimates
in space $H^k(\mathbb R^3)$ for the original function
$\tilde{u}^{\pm}$ will be the same as for $(\varphi^{\pm},
\phi^{\pm})$. Combining Proposition \ref{Proposition2.2},
\ref{Proposition2.3} and  \ref{Proposition2.4}, we finally deduce
that the solution $\tilde{U}$ to the IVP \eqref{2.11} has the
following decay rates in time.
\begin{Proposition}\label{Proposition2.5} Let $k>-\frac{3}{2}$ and assume
that the initial data $U_0\in L^1(\mathbb R^3)$, then for any $t\ge
0$, the global solution
$\tilde{{U}}=(\tilde{n}^+,\tilde{u}^+,\tilde{n}^-,\tilde{u}^-)^t$ of
the IVP \eqref{2.11} satisfies
\begin{equation}\label{2.45}\left\|\nabla^k\left(\tilde{n}^+, \tilde{n}^-\right)(t)\right\|_{L^2}\leq
C(1+t)^{-\frac{1}{4}-\frac{k}{2}}\left\|
{U}(0)\right\|_{L^{1}},\end{equation}
\begin{equation}\label{2.46}\left\|\nabla ^k\left(\tilde{u}^+, \tilde{u}^-, \beta_1\tilde{n}^+
+\beta_2\tilde{n}^-\right)\right\|_{L^2}\leq
C(1+t)^{-\frac{3}{4}-\frac{k}{2}}\| {U}(0)\|_{L^{1}}.\end{equation}
 If in addition, the initial data satisfies \eqref{2.39}, the following lower-bounds on convergence
 rate hold
\begin{equation}\min\left\{\left\|\tilde{n}^{+,l}(t)\right\|_{L^2},\left\|{\tilde{n}}^{-,l}(t)\right\|_{L^2}\right\}
\geq C_1c_0(1+t)^{-\frac{1}{4}},\label{2.47}
\end{equation}
\begin{equation}\min\left\{\left\|\tilde{u}^{+,l}(t)\right\|_{L^2},\left\|{\tilde{u}}^{-,l}(t)\right\|_{L^2},
\left\|\beta_1\tilde{n}^{+,l}
+\beta_2\tilde{n}^{-,l}(t)\right\|_{L^2} \right\}\geq
C_1c_0(1+t)^{-\frac{3}{4}},\label{2.48}
\end{equation}
if $t$ large enough.
\end{Proposition}

%%%%%%%%%%%%%%%%%%%%%%%%%%%%%%%%%%%%%%%%%%%%%%%
\section{\leftline {\bf{Energy estimates for the nonlinear system.}}}
\setcounter{equation}{0}
%%%%%%%%%%%%%%%%%%%%%%%%%%%%%%%%%%%%%%%%%%%%%%%
In this section, we devote ourselves to deriving the a priori energy
estimates for the nonlinear system \eqref{2.1}. To see this, we
assume a priori that for sufficiently small $\delta>0$,
\begin{equation}\label{3.1}\|\left(n^+, n^-\right)\|_{H^{\ell+1}}+\|\left(u^+, u^-\right)\|_{H^\ell}\le
\delta.
\end{equation} \par
In what follows, a series of lemmas on the energy estimates are
given. First, we deduce energy estimate on
$\left(n^+,u^+,n^-,u^-\right)$, which is stated in the following
lemma.
\smallskip
\begin{Lemma}\label{Lemma3.1}Assume that the notations and hypotheses  of Theorem \ref{1mainth} and \eqref{3.1}  are in force.
Then, for $0\leq k\leq \ell$, it holds that
\begin{equation}\label{L1.1}
\begin{split}
& \frac{1}{2}
 \frac{\rm d}{{\rm d}t}\left\{\|\nabla^k\left(\beta^+ n^++\beta^-n^-\right)\|_{L^2}^2+\frac{\sigma^+}{\beta_2}
 \|\nabla^{k+1}
 n^+\|_{L^2}^2+\frac{\sigma^-}{\beta_3} \|\nabla^{k+1} n^-\|_{L^2}^2
 + \frac{1}{\beta_2} \|\nabla^ku^+\|_{L^2}^2+\frac{1}{\beta_3} \|\nabla^ku^-\|_{L^2}^2
\right\}\\
&\quad+\displaystyle\frac{3}{4\beta_2}\Big(\nu^+_1\|\nabla^{k+1}
u^+\|_{L^2}^2+\nu^+_2\|\nabla^k\text{\rm div}
u^+\|_{L^2}^2\Big)+\frac{3}{4\beta_3}\Big(\nu^-_1\|\nabla^{k+1}
u^-\|_{L^2}^2+\nu^-_2\|\nabla^k\text{\rm div} u^-\|_{L^2}^2\Big)
\\
&\leq C\delta\Big(\|\nabla^{k+1} (n^+, n^-)\|_{H^1}^2 +\|\nabla^{k}
(u^+, u^-)\|_{L^2}^2\Big),
\end{split}
\end{equation}
for some constant $C>0$ independent of $\delta$, where
$\displaystyle\beta^+=\sqrt{\frac{\beta_1}{\beta_2}}$ and
$\displaystyle\beta^-=\sqrt{\frac{\beta_4}{\beta_3}}$.
\end{Lemma}
\begin{remark} Compared to Lemma \ref{2.2} and Lemma \ref{2.3} of
\cite{c1}, where $\|\nabla^{k} (n^+, n^-)\|_{H^2}^2$ are involved in
the right--hand side of the corresponding energy inequality,
\eqref{L1.1} is new and different since its right--hand side only
includes $\|\nabla^{k+1} (n^+, n^-)\|_{H^1}^2$, and particularly
excludes the term $\|\nabla^{k} (n^+, n^-)\|_{L^2}^2$. We remark
that this new type of energy inequality \eqref{L1.1} enables us to
close our energy estimates at each $k$th level and plays an
essential role in our analysis.
\end{remark}
\begin{proof}
For $0\leq k\leq \ell$, multiplying $\nabla^k\eqref{2.1}_1$,
$\nabla^k\eqref{2.1}_2$, $\nabla^k\eqref{2.1}_3$ and
$\nabla^k\eqref{2.1}_4$ by
$\displaystyle\frac{\beta_1}{\beta_2}\nabla^kn^+$,
$\displaystyle\frac{1}{\beta_2}\nabla^ku^+$,
$\displaystyle\frac{\beta_4}{\beta_3}\nabla^kn^-$ and
$\displaystyle\frac{1}{\beta_3}\nabla^ku^-$ respectively, summing up
and then integrating the resultant equation over $\mathbb{R}^3$ by
parts, we have
\begin{equation}\label{L1-1}
\begin{split}
& \frac{1}{2}
 \frac{\rm d}{{\rm d}t}\left\{\|\beta^+ \nabla^kn^+\|_{L^2}^2+\|\beta^-\nabla^kn^-\|_{L^2}^2
 +\frac{\sigma^+}{\beta_2} \|\nabla^{k+1}
 n^+\|_{L^2}^2\right.\\
 &\quad+\left.\frac{\sigma^-}{\beta_3} \|\nabla^{k+1} n^-\|_{L^2}^2
 + \frac{1}{\beta_2} \|\nabla^ku^+\|_{L^2}^2+\frac{1}{\beta_3} \|\nabla^ku^-\|_{L^2}^2
\right\}\\
&\quad+\displaystyle\frac{1}{\beta_2}\Big(\nu^+_1\|\nabla^{k+1}
u^+\|_{L^2}^2+\nu^+_2\|\nabla^k\text{\rm div}
u^+\|_{L^2}^2\Big)+\frac{1}{\beta_3}\Big(\nu^-_1\|\nabla^{k+1}
u^-\|_{L^2}^2+\nu^-_2\|\nabla^k\text{\rm div} u^-\|_{L^2}^2\Big)\\
 &=\displaystyle\left\langle\nabla^k F_1,
\frac{\beta_1}{\beta_2}\nabla^k n^+\right\rangle +
    \left\langle\nabla^k F_2,  \frac{1}{\beta_2}\nabla^k u^+\right\rangle+
      \left\langle\nabla^k F_3, \frac{\beta_4}{\beta_3}\nabla^kn^- \right\rangle +
        \left\langle \nabla^kF_4, \frac{1}{\beta_3}\nabla^k u^-
        \right\rangle\\&\quad-
          \left\langle \Delta\nabla^k n^+, \frac{\beta_1\sigma^+}{\beta_2}\nabla^kF_1 \right\rangle-
           \left\langle \Delta \nabla^kn^-, \frac{\beta_4\sigma^-}{\beta_3}\nabla^kF_3
                \right\rangle-\Big\langle \nabla^{k+1} n^-,  \nabla^ku^+ \Big\rangle -
            \Big\langle \nabla^{k+1} n^+,  \nabla^ku^-
            \Big\rangle\\&:=I_1^k +  I_2^k + I_3^k  +   I_4^k    + I_5^k
            +  I_6^k+I_7^k + I_8^k,
\end{split}
\end{equation}
where $I_j^k$, $j=1,2,\cdots 8$, denote the corresponding terms in
the above equation, which will be estimated as follows. Notice that
the nonlinear source terms $F_k$ $(k=1,2,3,4)$ possess the following
equivalent properties:
\begin{equation}\label{L1-2}
F_1
 \sim
  n^+\partial_k u^+_k
   +
    u^+_k \partial_k n^+,
\end{equation}
\begin{equation}\label{L1-3}
F_2^j
 \sim
  (n^+ +n^-)
   \partial_j n^+
    +
     (u^+\cdot \nabla)u^+_j
      +
       \partial_i u^+_j \partial_i n^+
        +
         \partial_j u^+_i \partial_i n^+
          +
           (n^++n^-)\partial_j n^-,
\end{equation}
\begin{equation}\label{L1-4}
F_3
 \sim
  n^-\partial_k u^-_k
   +
    u^-_k \partial_k n^-,
\end{equation}
\begin{equation}\label{L1-5}
F_4^j
 \sim
  (n^+ +n^-)
   \partial_j n^-
    +
     (u^-\cdot \nabla)u^-_j
      +
       \partial_i u^-_j \partial_i n^-
        +
        \partial_j u^-_i \partial_i n^-
        +
         (n^++n^-)\partial_j n^+.
\end{equation}
\par
Firstly, for $I^k_1$, it follows from integration by parts, Lemma
\ref{1interpolation}, Lemma \ref{es-product}, Young inequality and
 \eqref{3.1} that
\begin{equation}\label{L1-6}
\begin{split}
|I_1^k|
 &\leq C\|\nabla^k(n^+u^+)\|_{L^2}\|\nabla^{k+1} n^+\|_{L^2}\\
& \leq C\Big(\|\nabla^k n^+\|_{L^6}\|u^+\|_{L^3}+\|n^+\|_{L^3}\|\nabla^k u^+\|_{L^6}\Big)\|\nabla^{k+1} n^+\|_{L^2}\\
& \leq C\Big(\|\nabla^{k+1} n^+\|_{L^2}\|u^+\|_{H^1}+\|n^+\|_{H^1}\|\nabla^{k+1} u^+\|_{L^2}\Big)\|\nabla^{k+1} n^+\|_{L^2}\\
& \leq C\delta\Big(\|\nabla^{k+1} n^+\|_{L^2}^2+\|\nabla^{k+1}
u^+\|_{L^2}^2\Big).
\end{split}
\end{equation}
Similarly, for $I^k_3$, we have
\begin{equation}\label{L1-7}
|I_3^k|
 \leq C\delta\Big(\|\nabla^{k+1} n^-\|_{L^2}^2+\|\nabla^{k+1}
u^-\|_{L^2}^2\Big).
\end{equation}
Noting \eqref {L1-3} and employing similar arguments used in
\eqref{L1-6}, one has
\begin{equation}\label{L1-8}
\begin{split}
|I_2^k| \leq
 &
  C\Big|\Big\langle \nabla^k[(n^++n^-)\nabla n^+],\nabla^k u^+\Big\rangle\Big|
   +
    C\Big|\Big\langle \nabla^k[(u^+\cdot \nabla )u^+], \nabla^k u^+\Big\rangle\Big|\\
 &
     +
      C\Big|\Big\langle \nabla^k(\nabla n^+\cdot \nabla u^+),\nabla^k u^+\Big\rangle\Big|
       +
        C\Big|\Big\langle \nabla^k(\nabla n^+ (\nabla u^+)^\tau),\nabla^k u^+\Big\rangle\Big|\\
 &        +
           C\Big|\Big\langle \nabla^k[(n^++n^-)\nabla n^-],\nabla^k u^+\Big\rangle\Big|\\
 \leq
  &
   C\|\nabla^k u^+\|_{L^2}\Big(\|(n^++n^-)\|_{L^\infty}\|\nabla^{k+1} n^+\|_{L^2}+\|\nabla^k(n^++n^-)\|_{L^6}\|\nabla n^+\|_{L^3}\Big)\\
  &
    +
     C\|\nabla^k u^+\|_{L^6}\Big(\|u^+\|_{L^3}\|\nabla^{k+1} u^+\|_{L^2}+\|\nabla^k u^+\|_{L^2}\|\nabla u^+\|_{L^3}\Big)\\
  &
      +
       C\|\nabla^k u^+\|_{L^6}\Big(\|\nabla^{k+1}u^+\|_{L^2}\|\nabla n^+\|_{L^3}+\|\nabla u^+\|_{L^3}\|\nabla^{k+1} n^+\|_{L^2}\Big)\\
  &
        +
         C\|\nabla^k u^+\|_{L^2}\Big(\|(n^++n^-)\|_{L^\infty}\|\nabla^{k+1} n^-\|_{L^2}+\|\nabla^k(n^++n^-)\|_{L^6}\|\nabla n^-\|_{L^3}\Big)\\
\leq
 &
  C\delta \Big(\|\nabla^{k+1}(n^+, n^-)\|_{L^2}^2
    +
     \|\nabla^k u^+\|^2_{H^1}\Big).
\end{split}
\end{equation}
Similarly, for $I_4^k$, it holds that
\begin{equation}\label{L1-9}
 |I_4^k|
\leq
  C\delta \Big(\|\nabla^{k+1}(n^+, n^-)\|_{L^2}^2
    +
     \|\nabla^k u^-\|^2_{H^1}\Big).
\end{equation}
For the term $I_5^k$, it holds that
\begin{equation}\label{L1-10}
\begin{split}
|I_5^k|
 \leq
  &
   C\Big|\Big\langle\nabla^k[n^+\mathrm{div} u^++ u^+\cdot \nabla n^+],\Delta \nabla^k n^+\Big\rangle\Big|\\
 \leq
  &
   C\Big(\|n^+\|_{L^\infty}\|\nabla^{k+1} u^+\|_{L^2}+\|\nabla u^+\|_{L^3}\|\nabla^k n^+\|_{L^6}\Big)\|\nabla^{k+2} n^+\|_{L^2}\\
  &
    +
     C\Big(\|u^+\|_{L^\infty}\|\nabla^{k+1}n^+\|_{L^2}+\|\nabla n^+\|_{L^3}\|\nabla^k u^+\|_{L^6}\Big)\|\nabla^{k+2} n^+\|_{L^2}\\
 \leq
  &
   C\delta\Big(\|\nabla^{k+1} n^+\|_{H^1}^2
     +\|\nabla^{k+1} u^+\|_{L^2}^2\Big).
\end{split}
\end{equation}
Similarly, for $I_6^k$, we have
\begin{equation}\label{L1-11}
|I_6^k|
 \leq C\delta\Big(\|\nabla^{k+1} n^-\|_{H^1}^2
     + \|\nabla^{k+1} u^-\|_{L^2}^2\Big).
\end{equation}
For $I_7^k$ and $I_8^k$, it follows from integration by parts,
$\eqref{2.1}_1$, and $\eqref{2.1}_3$ that
\begin{equation}\begin{split}\label{L1-12}
I_7^k+I_8^k&=-\Big\langle \nabla^{k+1} n^-,  \nabla^ku^+ \Big\rangle
-\Big\langle \nabla^{k+1} n^+,  \nabla^ku^-\Big\rangle\\
 &=\Big\langle \nabla^{k} n^-,  \nabla^k\mathrm{div}u^+ \Big\rangle
+\Big\langle \nabla^{k} n^+,  \nabla^k\mathrm{div}u^-\Big\rangle\\
 &\displaystyle=-2\frac{d}{dt}\Big\langle \nabla^{k} n^-,  \nabla^kn^+
 \Big\rangle+\Big\langle \nabla^{k} n^-,  \nabla^kF^1\Big\rangle+
 \Big\langle \nabla^{k} n^+,  \nabla^kF^3
 \Big\rangle.
\end{split}\end{equation}
Furthermore, similar to the proofs of $I_1^k$ and $I_3^k$, the last
two terms on the right--hand side of \eqref{L1-12} can be estimated
as follows:
\begin{equation}\label{L1-13}
\Big|\Big\langle \nabla^{k} n^-,  \nabla^kF^1\Big\rangle\Big|+
 \Big|\Big\langle \nabla^{k} n^+,  \nabla^kF^3
 \Big\rangle\Big|\leq C\delta\Big( \|\nabla^{k+1} (n^+,n^-)\|_{L^2}^2
     +\|\nabla^{k+1} (u^+,u^-)\|_{L^2}^2\Big).
\end{equation}\par
Finally, substituting \eqref{L1-6}-\eqref{L1-13} into \eqref{L1-1}
and using the smallness of $\delta$, we get \eqref{L1.1} and thus
complete the proof of Lemma \ref{Lemma3.1}.
\end{proof}

Notice that \eqref{L1.1} only involves the dissipative estimates of
$u^\pm$. Next, we deduce the dissipative estimates for $n^\pm$,
which is stated in the following lemma.
\begin{Lemma}\label{Lemma3.3}
Assume that the notations and hypotheses  of Theorem \ref{1mainth}
and \eqref{3.1}  are in force. Then, for $0\leq k\leq \ell$, it holds
that
\begin{equation}\label{L3.1}
\begin{split}
& \frac{\rm d}{{\rm d}t}
 \left\{\left\langle \nabla^k u^+, \frac{1}{\beta_2}\nabla \nabla^k n^+\right\rangle
 +\left\langle\nabla^k u^-,\frac{1}{\beta_3}\nabla\nabla^k n^-\right\rangle
 \right\}\\
 &\quad+\|\nabla^{k+1}\left(\beta^+\nabla n^+ +\beta^-\nabla n^-\right)\|_{L^2}^2
+\frac{3\sigma^+}{4\beta_2}\|\nabla^{k+2} n^+\|_{L^2}^2+\frac{3\sigma^-}{4\beta_3}\|\nabla^{k+2} n^-\|_{L^2}^2\\
&\leq C\Big(\delta \|\nabla^{k+1} (n^+,n^-)\|_{L^2}^2+\|\nabla^{k+1}
(u^+,u^-)\|_{L^2}^2\Big),
\end{split}
\end{equation}
for some constant $C>0$ independent of $\delta$.
\end{Lemma}
\begin{proof}
For $0\leq k\leq\ell$, applying $\nabla^k$ to $\eqref{2.1}_2$,
$\eqref{2.1}_4$ and then multiplying the resultant equations by
$\displaystyle\frac{1}{\beta_2} \nabla^k \nabla n^+$ and
$\displaystyle\frac{1}{\beta_3} \nabla^k \nabla n^-$ respectively,
summing up and then integrating over $\mathbb{R}^3$, we obtain
\begin{equation}\label{L3-1}
\begin{split}
& \frac{\rm d}{{\rm d}t}
 \left\{\left\langle \nabla^ku^+, \frac{1}{\beta_2}\nabla\nabla^k n^+\right\rangle
    + \left\langle \nabla^ku^-, \frac{1}{\beta_3}\nabla\nabla^k n^-\right\rangle
 \right\}\\
&\quad +\|\nabla^{k+1}\left(\beta^+\nabla n^+ +\beta^-\nabla
n^-\right)\|_{L^2}^2 +\frac{\sigma^+}{\beta_2}\|\nabla^{k+2}
n^+\|_{L^2}^2+\frac{\sigma^-}{\beta_3}\|\nabla^{k+2}
n^-\|_{L^2}^2\\
&=\left\langle \nabla^k u^+,
\frac{1}{\beta_2}\partial_t\nabla\nabla^k n^+\right\rangle
   +\left\langle \Delta\nabla^k  u^+, \frac{\nu_1^+}{\beta_2}\nabla\nabla^k
n^+\right\rangle+\left\langle \nabla\nabla^k  {\rm div}u^+, \frac{\nu_2^+}{\beta_2}\nabla\nabla^k  n^+\right\rangle\\
&\quad + \left\langle\nabla^k  F_2, \frac{1}{\beta_2}\nabla \nabla^k
n^+\right\rangle + \left\langle \nabla^k
u^-,\frac{1}{\beta_3}\partial_t\nabla \nabla^k n^-\right\rangle +
            \left\langle \Delta \nabla^k u^-, \frac{\nu_1^-}{\beta_3}\nabla \nabla^k n^-\right\rangle\\
&\quad + \left\langle \nabla \nabla^k {\rm div}u^-,
\frac{\nu_2^-}{\beta_3}\nabla \nabla^k n^-\right\rangle +
\left\langle \nabla^k F_4, \frac{1}{\beta_3}\nabla \nabla^k n^-\right\rangle\\
&:= J_1^k + J_2^k + J_3^k + J_4^k + J_5^k + J_6^k + J_7^k + J_8^k.
\end{split}
\end{equation}
We shall estimate each term in the right hand side of \eqref{L3-1}.
First, it follows from integration by parts, $\eqref{2.1}_1$ and
$\eqref{2.1}_3$ that
\begin{equation}\label{L3-2}
\begin{split}
J_1^k+J_5^k&=\left\langle \nabla^k u^+,
\frac{1}{\beta_2}\partial_t\nabla\nabla^k
n^+\right\rangle+\left\langle \nabla^k u^-,
\frac{1}{\beta_3}\partial_t\nabla\nabla^k n^-\right\rangle\\
&=\left\langle \nabla^k{\rm div} u^+,
\frac{1}{\beta_2}\partial_t\nabla^k n^+\right\rangle+\left\langle
\nabla^k{\rm div} u^-,
\frac{1}{\beta_3}\partial_t\nabla^k n^-\right\rangle\\
&=\frac{1}{\beta_2}\|\nabla^k{\rm div}
u^+\|_{L^2}^2+\frac{1}{\beta_2}\|\nabla^k{\rm div} u^-\|_{L^2}^2\\
&\quad-\left\langle\nabla^k{\rm div} u^+, \frac{1}{\beta_2}\nabla^k
F_1\right\rangle-\left\langle \nabla^k{\rm div} u^-,
\frac{1}{\beta_3}\nabla^k F_3\right\rangle.
\end{split}
\end{equation}
Moreover, the last two terms in the right--hand of \eqref{L3-2} can
be bounded by
\begin{equation}\label{L3-3}
\begin{split}
\left|\left\langle\nabla^k{\rm div} u^+, \frac{1}{\beta_2}\nabla^k
F_1\right\rangle\right|&\leq
  C\Big(\|n^+\|_{L^\infty}\|\nabla^{k+1}  u^+\|_{L^2}
   +\|\nabla^k n^+\|_{L^6}\|\nabla  u^+\|_{L^3}\Big)
      \|\nabla^{k+1} u^+\|_{L^2}\\ &\quad+
      C\Big(\|u^+\|_{L^\infty}\|\nabla^{k+1} n^+\|_{L^2}
       +\|\nabla^k u^+\|_{L^6}\|\nabla  n^+\|_{L^3}\Big)\|\nabla^{k+1} u^+\|_{L^2}\\
       &\leq C\delta\Big(\|\nabla^{k+1} n^+\|_{L^2}^2 +\|\nabla^{k+1} u^+\|_{L^2}^2\Big),
\end{split}
\end{equation}
and similarly,
\begin{equation}\label{L3-4}
\begin{split}
\left|\left\langle\nabla^k{\rm div} u^-, \frac{1}{\beta_3}\nabla^k
F_3\right\rangle\right|\leq C\delta\Big(\|\nabla^{k+1} n^-\|_{L^2}^2
+\|\nabla^{k+1} u^-\|_{L^2}^2\Big).
\end{split}
\end{equation}
By virtue of integration by parts and Young inequality, we have
\begin{equation}\label{L3-5}
|J_2^k|+|J_3^k|+|J_6^k|+|J_7^k|\leq \varepsilon\|\nabla^{k+2}
(n^+,n^-)\|_{L^2}^2+C_\varepsilon\|\nabla^{k+1} (u^+,u^-)\|_{L^2}^2,
\end{equation}
where $\varepsilon$ is a sufficiently small positive constant which
will be determined later. Similar to the proof of $I_2^k$, for
$J_4^k$, it holds that
\begin{equation}\label{L3-6}
\begin{split}
|J_4^k| \leq
 &
  C\Big|\Big\langle \nabla^k[(n^++n^-)\nabla n^+],\nabla \nabla^k n^+\Big\rangle\Big|
   +
    C\Big|\Big\langle \nabla^k[(u^+\cdot \nabla )u^+], \nabla \nabla^k n^+\Big\rangle\Big|\\
 &
     +
      C\Big|\Big\langle \nabla^k(\nabla n^+\cdot \nabla u^+),\nabla \nabla^k n^+\Big\rangle\Big|
       +
        C\Big|\Big\langle \nabla^k(\nabla n^+( \nabla u^+)^\tau),\nabla \nabla^k n^+\Big\rangle\Big|\\
        &
          +
           C\Big|\Big\langle \nabla^k[(n^++n^-)\nabla n^-],\nabla \nabla^k n^+\Big\rangle\Big|\\
 \leq
  &
   C\|\nabla^{k+1} n^+\|_{L^2}\Big(\|(n^++n^-)\|_{L^3}\|\nabla^{k+1}n^+\|_{L^6}+
   \|\nabla^k(n^++n^-)\|_{L^6}\|\nabla n^+\|_{L^3}\Big)\\
  &
    +
     C\|\nabla^{k+1} n^+\|_{L^2}\Big(\|u^+\|_{L^\infty}\|\nabla^{k+1}u^+\|_{L^2}+\|\nabla^k u^+\|_{L^6}\|\nabla u^+\|_{L^3}\Big)\\
  &
      +
       C\|\nabla^{k+1} n^+\|_{L^6}\Big(\|\nabla^{k+1}  u^+\|_{L^2}\|\nabla n^+\|_{L^3}+\|\nabla u^+\|_{L^3}
       \|\nabla^{k+1} n^+\|_{L^2}\Big)\\
  &
        +
         C\|\nabla^{k+1} n^+\|_{L^2}\Big(\|(n^++n^-)\|_{L^3}\|\nabla^{k+1}n^-\|_{L^6}
         +\|\nabla^k(n^++n^-)\|_{L^6}\|\nabla n^-\|_{L^3}\Big)\\
\leq
 &
  C\delta\Big (\|\nabla^{k+1} (n^+,n^-) \|_{H^1}^2
       +
        \|\nabla^{k+1} u^+\|^2\Big).
\end{split}
\end{equation}
Similarly, we have
\begin{equation}\label{L3-7}
|J_8^k|
 \leq
  C\delta\Big (\|\nabla^{k+1} (n^+,n^-) \|_{H^1}^2
       +
        \|\nabla^{k+1} u^-\|^2\Big).
\end{equation}
\par Finally, putting \eqref{L3-2}--\eqref{L3-7} into \eqref{L3-1},
using the smallness of $\delta$ and choosing $\varepsilon$ small
enough, we get \eqref{L3.1}, and thus complete the proof of Lemma
\ref{Lemma3.3}.
\end{proof}
%%%%%%%%%%%%%%%%%%%%%%%%%%%%%%%%%%%%%%%%%%%%%%%
\section{\leftline {\bf{Proof of Theorem \ref{1mainth}.}}}
\setcounter{equation}{0}
%%%%%%%%%%%%%%%%%%%%%%%%%%%%%%%%%%%%%%%%%%%%%%%
\subsection{Proof of global existence and uniqueness}
In this subsection, we shall show global existence and uniqueness of
solutions stated in Theorem \ref{1mainth}. By virtue of the classic
local existence results in \cite{Mat1, Mat2} and the continuation in
time of the local solution, we see that to prove the global
existence result of Theorem 1.1, it suffices to close the a priori
assumption \eqref{3.1} and prove the energy estimate \eqref{1.17}.
To begin with, we define the following two time--weighted energy
functionals
\begin{equation}\label{4.1} E_k^{\ell}(t)=\sup\limits_{0\leq\tau\leq t}\Big
\{(1+\tau)^{\frac{3}{4}+\frac{k}{2}}\Big(\|\nabla^k(\beta^+n^++\beta^-n^-,u^+,u^-)(\tau)\|_{H^{\ell-k}}+
\|\nabla^{k+1}(n^+, n^-)(\tau)\|_{H^{\ell-k}}\Big)\Big \},
\end{equation}
for $0\leq k\leq \ell$, and
\begin{equation}\label{4.2} E_0(t)=\sup\limits_{0\leq\tau\leq t}\Big
\{(1+\tau)^{\frac{1}{4}} \|(n^+, n^-)(\tau)\|_{L^2}\Big \}.
\end{equation}
Choosing a sufficiently large positive constant $D_1$, and computing
$\eqref{L1.1}\times D_1+\eqref{L3.1}$, and then summing up the
resultant inequality from $k=0$ to $\ell$, we have from the
smallness of $\delta$ that
\begin{equation}\begin{split}\label{4.3}\frac{\mathrm{d}}{\mathrm{d}t}\mathcal
E_0^{\ell}(t)&+C\Big(\|\nabla(u^+,
u^-)(t)\|^2_{H^{\ell}}+\|\nabla^2(n^+,
n^-)(t)\|^2_{H^{\ell}}+\|\nabla(\beta^+n^++\beta^-n^-)(t)\|^2_{H^{\ell}}\Big)\\
\leq & C\delta\Big(\|\nabla(n^+, n^-)(t)\|^2_{L^{2}}+\|(u^+,
u^-)(t)\|^2_{L^{2}}\Big),
\end{split}\end{equation}
where
\begin{align*}
\mathcal E_0^{\ell}(t)&=\frac{D_1}{2} \frac{\rm d}{{\rm
d}t}\left\{\|\left(\beta^+
n^++\beta^-n^-\right)\|_{H^\ell}^2+\frac{\sigma^+}{\beta_2}
 \|\nabla n^+\|_{H^\ell}^2+\frac{\sigma^-}{\beta_3} \|\nabla n^-\|_{H^\ell}^2
 + \frac{1}{\beta_2} \|u^+\|_{H^\ell}^2+\frac{1}{\beta_3}
 \|u^-\|_{H^\ell}^2
\right\}\\
 &\quad+\displaystyle\sum_{k=0}^\ell\left\{\left\langle \nabla^k u^+, \frac{1}{\beta_2}\nabla \nabla^k n^+\right\rangle
 +\left\langle\nabla^k u^-,\frac{1}{\beta_3}\nabla\nabla^k
 n^-\right\rangle\right\},
\end{align*}
which is equivalent to
$\|(\beta^+n^++\beta^-n^-)(t)\|^2_{H^{\ell}}+\|\nabla(n^+,
n^-)(t)\|^2_{H^{\ell}}+\|(u^+, u^-)(t)\|^2_{H^{\ell}}$ since $D_1$
is large enough. Then, \eqref{4.3} together with Lemma \ref{lh2}
gives
 \begin{equation}\label{4.4}\frac{\mathrm{d}}{\mathrm{d}t}\mathcal
E_0^{\ell}(t)+D_2\mathcal E_0^{\ell}(t)\le
 C\Big(\|(\beta^+n^{+,l}+\beta^-n^{-,l})(t)\|^2_{L^2}+\|\nabla( n^{+,l}, n^{-,l})(t)\|^2_{L^2}+\|( u^{+,l}, u^{-,l})(t)\|^2_{L^2}
\Big),
\end{equation}
where $D_2$ is a positive constant independent of $\delta$.
\par Defining $U=(n^+, u^+, n^-, u^-)^t$ and $\mathcal
F=(F^1,F^2,F^3,F^4)^t$, it follows from Duhamel's principle that
\begin{equation}\label{4.5} U^l=\text{e}^{t\mathcal{B}}U^l(0)+\int_0^t\text{e}^{(t-\tau)\mathcal{B}}\mathcal
F^l(\tau)\mathrm{d}\tau,
\end{equation}
which together with Plancherel theorem, integration by parts,
Proposition \ref{2.5}, Lemma \ref{lh1}, \eqref{4.1} and \eqref{4.2}
implies
\begin{equation}\begin{split}\label{4.6}
\|(u^{+,l}, u^{-,l})\|_{L^2}&\leq C (1+t)^{-\frac{3}{4}}\|
U(0)\|_{L^1}+\int_0^t(1+t-\tau)^{-\frac{5}{4}}
 \|\mathcal(n^+u^+, n^-u^-)(\tau)\|_{L^1}\mathrm{d}\tau\\
 &\quad+\int_0^t
 \|\text{e}^{(t-\tau)\mathcal{B}}(F_2, F_4)(\tau)\|_{L^2}\mathrm{d}\tau\\
&\leq CK_0(1+t)^{-\frac{3}{4}}+\int_0^t(1+t-\tau)^{-\frac{5}{4}}
 (1+\tau)^{-1}E_0(t){E}_0^\ell(t)\mathrm{d}\tau\\
&\quad+\int_0^t
 \|\text{e}^{(t-\tau)\mathcal{B}}(F_2, F_4)(\tau)\|_{L^2}\mathrm{d}\tau\\
&\leq
(1+t)^{-\frac{3}{4}}\left(CK_0+E_0(t){E}_0^\ell(t)\right)+\int_0^t
 \|\text{e}^{(t-\tau)\mathcal{B}}(F_2, F_4)(\tau)\|_{L^2}\mathrm{d}\tau.
\end{split}\end{equation} As mentioned before, the strongly coupling terms
like $g_+\left(n^{+}, n^{-}\right)
\partial_{i} n^{+}+\bar{g}_{+}\left(n^{+}, n^{-}\right) \partial_{i}
n^{-}$ in \eqref{2.3} and $g_-\left(n^{+}, n^{-}\right)
\partial_{i} n^{-}+\bar{g}_{-}\left(n^{+}, n^{-}\right) \partial_{i}
n^{+}$ in \eqref{2.5} devote the slowest time--decay rates to the
third term on the right--side of \eqref{4.6}, which prevents us from
deriving the desired decay rates of $\|(u^{+,l}, u^{-,l})\|_{L^2}$.
To overcome this difficulty, the key idea here is to make full use
of good properties of $\beta^+n^{+}+\beta^-n^{-}$, and cleverly
rewrite them  as follows:
\begin{align}\label{4.7}
&g_+\left(n^{+}, n^{-}\right)
\partial_{i} n^{+}+\bar{g}_{+}\left(n^{+}, n^{-}\right) \partial_{i}
n^{-}\notag\\
&=\left(\frac{\mathcal{C}^2\rho^-(n^++1,n^-+1)}{\rho^+(n^++1,
n^-+1)}-\beta_1\right)\partial_i
n^++\left(\mathcal{C}^2(n^++1,n^-+1)-\beta_2\right)\partial_i n^-\notag\\
&=\frac{1}{\beta_1}\left(\frac{\mathcal{C}^2\rho^-(n^++1,n^-+1)}{\rho^+(n^++1,
n^-+1)}-\beta_1\right)\left(\beta_1\partial_i n^++\beta_2\partial_i
n^-\right)\notag\\
&\quad-\mathcal{C}^2(n^++1,n^-+1)\left[\left(\frac{\bar{\rho}^+\rho^-(n^++1,n^-+1)}
{\bar{\rho}^-\rho^+(n^++1,n^-+1)}-1\right)\right]\partial_i n^-\notag\\
&=\partial_i\left[\frac{1}{\beta_1}\left(\frac{\mathcal{C}^2\rho^-(n^++1,n^-+1)}{\rho^+(n^++1,
n^-+1)}-\beta_1\right)\left(\beta_1n^++\beta_2
n^-\right)\right]\\
&\quad-
\partial_i\left[\frac{1}{\beta_1}\left(\frac{\mathcal{C}^2\rho^-(n^++1,n^-+1)}{\rho^+(n^++1,
n^-+1)}-\beta_1\right)\right]\left(\beta_1n^++\beta_2
n^-\right)\notag\\
&\quad-\mathcal{C}^2(n^++1,n^-+1)\left(\frac{\bar{\rho}^+\rho^-(n^++1,n^-+1)}
{\bar{\rho}^-\rho^+(n^++1,n^-+1)}-1\right)\partial_i n^-\notag\\
&=-\partial_i {G}^i_2-
\partial_i\left[\frac{1}{\beta_1}\left(\frac{\mathcal{C}^2\rho^-(n^++1,n^-+1)}{\rho^+(n^++1,
n^-+1)}-\beta_1\right)\right]\left(\beta_1n^++\beta_2
n^-\right)\notag\\
&\quad-\mathcal{C}^2(n^++1,n^-+1)\left(\frac{\bar{\rho}^+\rho^-(n^++1,n^-+1)}
{\bar{\rho}^-\rho^+(n^++1,n^-+1)}-1\right)\partial_i n^-,\notag
\end{align}
and
\begin{align}\begin{split}\label{4.8}
&g_-\left(n^{+}, n^{-}\right)
\partial_{i} n^{-}+\bar{g}_{-}\left(n^{+}, n^{-}\right) \partial_{i}
n^{+}\\
&=\left(\frac{\mathcal{C}^2\rho^+(n^++1,n^-+1)}{\rho^-(n^++1,
n^-+1)}-\beta_4\right)\partial_i
n^-+\left(\mathcal{C}^2(n^++1,n^-+1)-\beta_3\right)\partial_i n^+\\
&=\frac{1}{\beta_4}\left(\frac{\mathcal{C}^2\rho^+(n^++1,n^-+1)}{\rho^-(n^++1,
n^-+1)}-\beta_4\right)\left(\beta_3\partial_i n^++\beta_4\partial_i
n^-\right)\\
&\quad-\mathcal{C}^2(n^++1,n^-+1)\left(\frac{\bar{\rho}^-\rho^+(n^++1,n^-+1)}
{\bar{\rho}^+\rho^-(n^++1,n^-+1)}-1\right)\partial_i n^+\\
&=\partial_i\left[\frac{1}{\beta_4}\left(\frac{\mathcal{C}^2\rho^+(n^++1,n^-+1)}{\rho^-(n^++1,
n^-+1)}-\beta_4\right)\left(\beta_3n^++\beta_4
n^-\right)\right]\\
&\quad-\partial_i\left[\frac{1}{\beta_4}\left(\frac{\mathcal{C}^2\rho^+(n^++1,n^-+1)}{\rho^-(n^++1,
n^-+1)}-\beta_4\right)\right]\left(\beta_3 n^++\beta_4 n^-\right)\\
&\quad-\mathcal{C}^2(n^++1,n^-+1)\left(\frac{\bar{\rho}^-\rho^+(n^++1,n^-+1)}
{\bar{\rho}^+\rho^-(n^++1,n^-+1)}-1\right)\partial_i n^+\\
&=-\partial_i G_4^i
-\partial_i\left[\frac{1}{\beta_4}\left(\frac{\mathcal{C}^2\rho^+(n^++1,n^-+1)}{\rho^-(n^++1,
n^-+1)}-\beta_4\right)\right]\left(\beta_3 n^++\beta_4 n^-\right)\\
&\quad-\mathcal{C}^2(n^++1,n^-+1)\left(\frac{\bar{\rho}^-\rho^+(n^++1,n^-+1)}
{\bar{\rho}^+\rho^-(n^++1,n^-+1)}-1\right)\partial_i n^+.
\end{split}\end{align}
On the other hand, by noticing the pressure differential $dP$, we
have
\begin{align}\label{4.9}
\nabla P&=\mathcal{C}^2(n^++1,n^-+1)\big[\rho^-(n^++1,n^-+1)\nabla n^++\rho^+(n^++1,n^-+1)\nabla n^-\big]\notag\\
&=\mathcal{C}^2(n^++1,n^-+1)\big[\rho^-(n^++1,n^-+1)-\bar{\rho}^-\big]\nabla
n^+\notag\\
&\quad+\mathcal{C}^2(n^++1,n^-+1)\big[\rho^+(n^++1,n^-+1)-\bar{\rho}^+\big]\nabla n^-\\
&\quad+\big[\mathcal{C}^2(n^++1,n^-+1)-\mathcal{C}^2(1,1)\big]\big[\bar{\rho}^-\nabla
n^++\bar{\rho}^+\nabla
n^-\big]\notag\\
&\quad+\mathcal{C}^2(1,1)\big(\bar{\rho}^-\nabla
n^++\bar{\rho}^+\nabla n^-\big).\notag
\end{align}
This together with H\"{o}lder inequality, Lemma
\ref{1interpolation}, \eqref{3.1}, \eqref{4.1} and \eqref{4.2}
implies
\begin{align}\begin{split}\nonumber
\|\nabla P\|_{L^2}&\leq C\left(\|\nabla(n^+,
n^-)\|_{L^3}\left(\|\rho^+-\bar{\rho}^+\|_{L^6}+\|\rho^--\bar{\rho}^-\|_{L^6}\right)+\|\bar{\rho}^-\nabla
n^++\bar{\rho}^+\nabla n^-\|_{L^2}\right)\\
&\leq C\left(\|\nabla(n^+, n^-)\|_{H^1}\|\nabla
P\|_{L^2}+(1+t)^{-\frac{3}{4}}E_0^\ell(t)\right)\\
&\leq C\left(\delta\|\nabla
P\|_{L^2}+(1+t)^{-\frac{3}{4}}E_0^\ell(t)\right),
\end{split}\end{align}
which together with the smallness of $\delta$ gives
\begin{equation}\label{4.10}
\|\nabla(P, \rho^+, \rho^+)\|_{L^2}\leq
C(1+t)^{-\frac{3}{4}}E_0^\ell(t).
\end{equation}
Therefore, denoting $\bar{P}=P(1,1)$ by the equilibrium state of
pressure $P$, we have from embedding estimate of Riesz potential
that
\begin{align}\begin{split}\label{4.11}
&\left\|P-\bar{P}\right\|_{L^2}\\
&\lesssim\|\Lambda^{-1}\nabla P\|_{L^2}\\
&\lesssim \left(\|\mathcal{C}^2\big(\rho^--\bar{\rho}^-\big)\nabla
n^+\right\|_{L^{\frac{6}{5}}}+\left\|\mathcal{C}^2\big(\rho^+-\bar{\rho}^+\big)\nabla
n^-\right\|_{L^{\frac{6}{5}}}\\
&\quad+\left\|\big(\mathcal{C}^2-\mathcal{C}^2(1,1)\big)\big(\bar{\rho}^-\nabla
n^++\bar{\rho}^+\nabla
n^-\big)\right\|_{L^{\frac{6}{5}}}\\
&\quad+\left(\|\mathcal{C}^2(1,1)\big(\bar{\rho}^- n^++\bar{\rho}^+
n^-\big)\right\|_{L^2}\\
&\lesssim
\left\|\big(\rho^--\bar{\rho}^-\big)\right\|_{L^3}\left\|\nabla
n^+\right\|_{L^{2}}+\left\|\big(\rho^+-\bar{\rho}^+\big)\right\|_{L^3}\left\|\nabla
n^-\right\|_{L^{2}}\\
&\quad+\left\|\big(\mathcal{C}^2-\mathcal{C}^2(1,1)\big)\right\|_{L^3}\left\|\bar{\rho}^-
\nabla n^++\bar{\rho}^+\nabla
n^-\right\|_{L^{2}}+\left\|\bar{\rho}^-
n^++\bar{\rho}^+n^-\right\|_{L^2}\\
 &\lesssim (1+t)^{-\frac{3}{4}}E_0^\ell(t),
\end{split}\end{align}
which particularly gives
\begin{equation}\label{4.12}
\left\|(\rho^+-\bar{\rho}^+, \rho^--\bar{\rho}^-)\right\|_{L^2}
\lesssim(1+t)^{-\frac{3}{4}}E_0^\ell(t).
\end{equation}
Consequently, combining \eqref{4.1}, \eqref{4.7}, \eqref{4.8},
\eqref{4.12} and using the fact that $\frac{\bar{\rho}^\pm\rho^\mp}
{\bar{\rho}^{\mp}\rho^\pm}-1\sim
\rho^+-\bar{\rho}^++\rho^--\bar{\rho}^-$, we conclude that
\begin{align}\label{4.13}
&\int_0^t
 \|\text{e}^{(t-\tau)\mathcal{B}}(F_2,
 F_4)(\tau)\|_{L^2}\mathrm{d}\tau\nonumber\\
 &\leq C\int_0^t(1+t-\tau)^{-\frac{5}{4}} \|(G_2,
 G_4)(\tau)\|_{L^2}\mathrm{d}\tau+C\int_0^t(1+t-\tau)^{-\frac{3}{4}} \|(F_2-\text{div} G_2,
 F_4-\text{div} G_4)(\tau)\|_{L^2}\mathrm{d}\tau\nonumber\\
&\quad +C\int_0^t(1+t-\tau)^{-\frac{5}{4}} \|(n^+,
n^-)(\tau)\|_{L^2}\|(\beta^+n^++\beta^-n^-)(\tau)\|_{L^2}\mathrm{d}\tau\nonumber\\
&\quad+ C\int_0^t(1+t-\tau)^{-\frac{3}{4}} \left(\|(u^+,
u^-)(\tau)\|_{L^2}+\|\nabla(n^+,
n^-)(\tau)\|_{L^2}\right)\left(\|\nabla(u^+,
u^-)(\tau)\|_{L^2}\right.\nonumber\\
&\quad\left.+\|\nabla(n^+,
n^-)(\tau)\|_{L^2}\|(\beta^+n^++\beta^-n^-)(\tau)\|_{L^2}\right)\\
&\quad+\|(\rho^+-\bar{\rho}^+,
\rho^--\bar{\rho}^-)(\tau)\|_{L^2}\left(\|\nabla^2(u^+,
u^-)(\tau)\|_{L^2}+\|\nabla(n^+,
n^-)(\tau)\|_{L^2}\right)\mathrm{d}\tau\nonumber\\
&\leq
 C\int_0^t\left((1+t-\tau)^{-\frac{5}{4}}(1+\tau)^{-1}E_0(t){E}_0^\ell(t)+(1+t-\tau)^{-\frac{3}{4}}(1+\tau)^{-\frac{3}{2}}
 \left(E_0(t)+{E}_0^\ell(t)\right){E}_0^\ell(t)\right)\mathrm{d}\tau\nonumber\\
&\leq
C(1+t)^{-\frac{3}{4}}\left[E_0(t){E}_0^\ell(t)+\left({E}_0^\ell(t)\right)^2\right],\nonumber
\end{align}
where $G_2=(G_2^1, G_2^2, G_2^3)^t$ and $G_4=(G_4^1, G_4^2,
G_4^3)^t$.
 Substituting \eqref{4.13} into \eqref{4.6} yields
\begin{equation}\label{4.14}
\left\|(u^{+,l}, u^{-,l})\right\|_{L^2}\lesssim
(1+t)^{-\frac{3}{4}}\left[K_0+E_0(t)E_0^\ell(t)+\left(E_0^\ell(t)\right)^2\right].
\end{equation}
Similarly, we also have
\begin{equation}\label{4.15}
\|(\beta^+n^{+,l}+\beta^-n^{-,l})(t)\|_{L^2}+\|\nabla( n^{+,l},
n^{-,l})(t)\|_{L^2}\lesssim
(1+t)^{-\frac{3}{4}}\left[K_0+E_0(t)E_0^\ell(t)+\left(E_0^\ell(t)\right)^2\right],
\end{equation} and
\begin{equation}\label{4.16}
\|(n^{+,l}, n^{-,l})(t)\|_{L^2}\lesssim
(1+t)^{-\frac{1}{4}}\left[K_0+E_0(t)E_0^\ell(t)+\left(E_0^\ell(t)\right)^2\right].
\end{equation}
Substituting \eqref{4.14}--\eqref{4.16} into \eqref{4.4} yields
 \begin{equation}\nonumber\frac{\mathrm{d}}{\mathrm{d}t}\mathcal E_0^{\ell}(t)+{D_2}\mathcal E_0^{\ell}(t)
 \le C(1+t)^{-\frac{3}{2}}\left[K_0+E_0(t)E_0^\ell(t)+\left(E_0^\ell(t)\right)^2\right]^2.
\end{equation}
Applying Gronwall's inequality to the above inequality, we can infer
that
\begin{equation}\begin{split}\nonumber\mathcal E_0^{\ell}(t)\leq &~\text{e}^{-{D_2}t}\mathcal E_0^{\ell}(0)
+C\int_0^t\text{e}^{-{D_2}(t-\tau)}(1+\tau)^{-\frac{3}{2}}\left[K_0+E_0(t)E_0^\ell(t)+\left(E_0^\ell(t)\right)^2\right]^2\mathrm{d}\tau\\
\leq&
~C(1+t)^{-\frac{3}{2}}\left[K_0+E_0(t)E_0^\ell(t)+\left(E_0^\ell(t)\right)^2\right]^2,
 \end{split}\end{equation}
  which together with \eqref{4.1} implies that
 \begin{equation}\label{4.17}E_0^{\ell}(t)\leq
 C\left[K_0+E_0(t)E_0^\ell(t)+\left(E_0^\ell(t)\right)^2\right].\end{equation}
Next, we deal with $E_0(t)$. By virtue of Lemma \ref{lh2},
\eqref{4.1}, \eqref{4.16} and \eqref{4.17}, we have
\begin{equation}\begin{split}\nonumber\|(n^+, n^-)\|_{L^2}&\leq
C\left(\|(n^{+,l}, n^{-,l})\|_{L^2}+\|(n^{+,h},
n^{-,h})\|_{L^2}\right)\\
&\leq C\left(\|(n^{+,l}, n^{-,l})\|_{L^2}+\|\nabla(n^{+},
n^{-})\|_{L^2}\right)\\
&\leq C
(1+t)^{-\frac{1}{4}}\left[K_0+E_0(t)E_0^\ell(t)+\left(E_0^\ell(t)\right)^2\right],
 \end{split}\end{equation}
which leads to
\begin{equation}\label{4.18}E_0(t)\leq
 C\left[K_0+E_0(t)E_0^\ell(t)+\left(E_0^\ell(t)\right)^2\right].\end{equation}
\par Finally, combining \eqref{4.17} with \eqref{4.18} and using
\eqref{1.16}, we conclude that
\begin{equation}\label{4.19}E_0^\ell(t)+E_0(t)\leq CK_0.\end{equation}
By a standard continuity argument, this closes the a priori
estimates \eqref{3.1} immediately since $\delta_0$ is sufficiently
small. This in turn allows us to integrate \eqref{4.4} directly in
time, to obtain
\begin{align*}
\|\nabla(n^+, n^-)(t)\|_{H^{\ell}}^2&+\|(u^+,u^-)(t)\|_{H^{\ell}}^2+
\|(\beta^+n^++\beta^-n^-)(t)\|_{H^{\ell}}^2\\
&\displaystyle+\int_0^t\Big(\|\nabla(n^+,
n^-)(\tau)\|_{H^{\ell}}^2+\|(u^+,u^-)(\tau)\|_{H^{\ell}}^2+
\|(\beta^+n^++\beta^-n^-)(\tau)\|_{H^{\ell}}^2\Big)\textrm{d}\tau\leq
CK_0^2,\end{align*} which together with \eqref{4.19} implies
\eqref{1.17} immediately. Therefore, we have completed the global
existence result of Theorem \ref{1.1}.
\subsection{Proof of upper bounds on decay rates}
In this subsection, we devote ourselves to proving the upper optimal
convergence rate of the solution stated in \eqref{1.18}-\eqref{1.21}
of Theorem \ref{1mainth}. We first show \eqref{1.19}-\eqref{1.21}.
Noticing \eqref{4.1} and \eqref{4.19}, it suffices to prove that
$E_j^{\ell}(t)\leq CK_0$, for any $1\le j\le \ell$. We will make
full use of the low--frequency and high--frequency decomposition,
and employ key linear convergence estimates obtained in Section 2
and uniform nonlinear energy estimates obtained in Section 3 to
achieve this goal by induction.

\begin{Theorem}\label{3.2mainth} Assume that the  hypotheses of Theorem \ref{1mainth}
and \eqref{1.20} are in force. Then there exists a positive constant
$C$ independent of $t$, such that
\begin{equation}\nonumber E_j^{\ell}(t)\le CK_0,
\end{equation}
for $1\le j\le \ell$.
\end{Theorem}
\begin{proof} We will employ mathematical inductive method to prove
Theorem \ref{3.2mainth}. Therefore, by noticing \eqref{4.1} and
\eqref{4.19},
 it suffices to prove
the following Lemma \ref{Lemma3.2}. Thus, the proof Theorem
\ref{3.2mainth} is completed.
\end{proof}

\begin{Lemma}\label{Lemma3.2} Assume that the  hypotheses of Theorem \ref{1mainth} and \eqref{1.20}.
If additionally  \begin{equation}\label{4.20}E_{k-1}^{\ell}(t)\le
CK_0,\end{equation} then it holds that
\begin{equation}\label{4.21}E_k^{\ell}(t)\le CK_0,
\end{equation}
for $1\le k\le \ell.$
\end{Lemma}
\begin{proof}

We will combine the key linear estimates with delicate nonlinear
energy estimates based on good properties of the low--frequency and
high--frequency decomposition to prove Theorem \ref{3.2mainth}, and
the process involves the following three steps.\par {Step 1.} Energy
estimates on
$\|\nabla^j(\beta^+n^++\beta^-n^-)\|_{H^{\ell-j}}^2+\|\nabla^{j+1}(n^+,
n^-)\|_{H^{\ell-j}}^2+\|\nabla^j( u^+,u^-)\|_{H^{\ell-j}}^2$.
Choosing a sufficiently large positive constant $D_3$, and computing
$\eqref{L1.1}\times D_3+\eqref{L3.1}$, and then summing up the
resultant inequality from $k=j$ to $\ell$, we have from the
smallness of $\delta$ that
\begin{equation}\begin{split}\label{4.22}\frac{\mathrm{d}}{\mathrm{d}t}\mathcal
E_j^{\ell}(t)&+C\Big(\|\nabla^{j+1}(u^+,
u^-)(t)\|^2_{H^{\ell-j}}+\|\nabla^{j+2}(n^+,
n^-)(t)\|^2_{H^{\ell-j}}+\|\nabla^{j+1}(\beta^+n^++\beta^-n^-)(t)\|^2_{H^{\ell-j}}\Big)\\
\leq & C\delta\Big(\|\nabla^{j+1}(n^+,
n^-)(t)\|^2_{L^{2}}+\|\nabla^j(u^+, u^-)(t)\|^2_{L^{2}}\Big),
\end{split}\end{equation}
where
\begin{equation}\begin{split}\notag
\mathcal E_j^{\ell}&=\frac{D_3}{2} \frac{\rm d}{{\rm
d}t}\left\{\|\nabla^j\left(\beta^+
n^++\beta^-n^-\right)\|_{H^{\ell-j}}^2+\frac{\sigma^+}{\beta_2}
 \|\nabla^{j+1} n^+\|_{H^{\ell-j}}^2+\frac{\sigma^-}{\beta_3} \|\nabla^{j+1} n^-\|_{H^{\ell-j}}^2\right.\\
 &\quad+\left.\frac{1}{\beta_2} \|\nabla^ju^+\|_{H^{\ell-j}}^2+\frac{1}{\beta_3}
 \|\nabla^j u^-\|_{H^{\ell-j}}^2
\right\}+\displaystyle\sum_{k=j}^\ell\left\{\left\langle \nabla^k
u^+, \frac{1}{\beta_2}\nabla \nabla^k n^+\right\rangle
 +\left\langle\nabla^k u^-,\frac{1}{\beta_3}\nabla\nabla^k
 n^-\right\rangle\right\},
\end{split}\end{equation}
which is equivalent to
$\|\nabla^j(\beta^+n^++\beta^-n^-)(t)\|^2_{H^{\ell-j}}+\|\nabla^{j+1}(n^+,
n^-)(t)\|^2_{H^{\ell-j}}+\|\nabla^j(u^+, u^-)(t)\|^2_{H^{\ell-j}}$
since $D_3$ is large enough. Then, \eqref{4.22} together with Lemma
\ref{lh2} yields
 \begin{equation}\label{4.23}\frac{\mathrm{d}}{\mathrm{d}t}\mathcal
E_j^{\ell}(t)+D_4\mathcal E_j^{\ell}(t)\le
 C\left(\|\nabla^j(\beta^+n^{+,l}+\beta^-n^{-,l})\|^2_{H^{\ell-j}}+\|\nabla^{j+1}(n^{+,l},
n^{-,l})\|^2_{H^{\ell-j}}+\|\nabla^j( u^{+,l},
u^{-,l})\|^2_{H^{\ell-j}} \right),
\end{equation} where $D_4$ is a positive constant
independent of $\delta$.\par {Step 2.} Decay estimates on
$\|\nabla^j(\beta^+n^{+,l}+\beta^-n^{-,l})\|^2_{H^{\ell-j}}+\|\nabla^{j+1}(n^{+,l},
n^{-,l})\|^2_{H^{\ell-j}}+\|\nabla^j( u^{+,l},
u^{-,l})\|^2_{H^{\ell-j}}$. To begin with, we deal with $\|\nabla^j(
u^{+,l}, u^{-,l})\|^2_{H^{\ell-j}}$. To do this, using \eqref{4.1}
and the inequality above \eqref{4.10}, we have
\begin{equation}\label{4.24}
\|\nabla(P, \rho^+, \rho^+)\|_{L^2}\leq
C(1+t)^{-\frac{5}{4}}E_1^\ell(t).
\end{equation}
For $2\leq j\leq\ell$, by employing \eqref{4.9}, H\"{ol}lder
inequality, Lemma \ref{es-product}, \eqref{4.1}, \eqref{4.19} and
\eqref{4.20}, we have
\begin{align*}
\left\|\nabla^j P\right\|_{L^2}&\leq
\left\|\nabla^{j-1}\left(\mathcal{C}^2(\rho^--\bar{\rho}^-)\nabla
n^+\right)\right\|_{L^2}+\left\|\nabla^{j-1}\left(\mathcal{C}^2(\rho^+-\bar{\rho}^+)\nabla
n^-\right)\right\|_{L^2}\\
&\quad+\left\|\nabla^{j-1}\left(\mathcal{C}^2(\bar{\rho}^-\nabla
n^++\bar{\rho}^+\nabla n^-)\right)\right\|_{L^2}\\
&\lesssim \left(\left\|\nabla^{j-1}\left(\rho^+,
\rho^-\right)\right\|_{L^6}\left\|\mathcal{C}^2\right\|_{L^3}
+\left\|\left(\rho^+-\bar{\rho}^+,\rho^--\bar{\rho}^-\right)\right\|_{L^3}\left\|\nabla
^{j-1}\mathcal{C}^2\right\|_{L^6} \right)\left\|\nabla(n^+,
n^-)\right\|_{L^\infty}\\
&\quad+\left\|\mathcal{C}^2\right\|_{L^\infty}\left\|\left({\rho}^+
-\bar{\rho}^+, {\rho}^-
-\bar{\rho}^-\right)\right\|_{L^2}\left\|\nabla^j(n^+,n^-)\right\|_{L^2}\\
&\quad+\left\|\mathcal{C}^2\right\|_{L^\infty}\left\|\nabla^{j}\left(\bar{\rho}^-
n^++\bar{\rho}^+
n^-\right)\right\|_{L^2}+\left\|\nabla^{j-1}\mathcal{C}^2\right\|_{L^6}\left\|\left(\bar{\rho}^-
\nabla n^++\bar{\rho}^+ \nabla n^-\right)\right\|_{L^3}\\
&\lesssim \left(K_0\left\|\nabla^{j}P\right\|_{L^2}
+\left\|\left(\rho^+-\bar{\rho}^+,\rho^--\bar{\rho}^-\right)\right\|_{H^1}\left\|\nabla
^{j}(n^+, n^-)\right\|_{L^2}\left\|\nabla(n^+,
n^-)\right\|_{H^2} \right)\\
&\quad+\left\|\nabla^{j}\left(\bar{\rho}^- n^++\bar{\rho}^+
n^-\right)\right\|_{L^2}+\left\|\nabla^{j}(n^+,
n^-)\right\|_{L^2}\left\|\left(\bar{\rho}^-
\nabla n^++\bar{\rho}^+ \nabla n^-\right)\right\|_{H^1}\\
&\lesssim K_0 \left\|\nabla^{j}P\right\|_{L^2}+K_0
(1+t)^{-1-\frac{j}{2}}+(1+t)^{-\frac{3}{4}-\frac{j}{2}}E_j^\ell(t),
\end{align*}
which together with the smallness of $K_0$ implies
\begin{equation}\label{4.25}
\left\|\nabla^{j}\left(P, \rho^+, \rho^-\right)\right\|_{L^2}\leq C
(1+t)^{-\frac{3}{4}-\frac{j}{2}}\left[K_0+E_j^\ell(t)\right].
\end{equation}
With \eqref{4.24} and \eqref{4.25} in hand, we can borrow similar
arguments used in \eqref{4.6} to get
\begin{align}\label{4.26}
\left\|\nabla^j(u^{+,l}, u^{-,l})\right\|_{L^2}&\lesssim
K_0(1+t)^{-\frac{3}{4}-\frac{j}{2}}+\int_0^{\frac{t}{2}}(1+t-\tau)^{-\frac{3}{4}-\frac{j+1}{2}}
 \left\|(n^+u^+, n^-u^-)(\tau)\right\|_{L^1}\mathrm{d}\tau\nonumber\\
&\quad+\int_{\frac{t}{2}}^t(1+t-\tau)^{-\frac{5}{4}}
 \left\||\xi|^j\left(\widehat{n^+u^+}, \widehat{n^-u^-}\right)(\tau)\right\|_{L^\infty}\mathrm{d}\tau\nonumber\\
&\quad+\int_0^{\frac{t}{2}}(1+t-\tau)^{-\frac{3}{4}-\frac{j+1}{2}}
 \left\|(G_2, G_4)(\tau)\right\|_{L^1}\mathrm{d}\tau\nonumber\\
&\quad+\int_{\frac{t}{2}}^t(1+t-\tau)^{-\frac{5}{4}}
 \left\||\xi|^j\left(\hat{G_2}, \hat{G_4}\right)(\tau)\right\|_{L^\infty}\mathrm{d}\tau\\
&\quad+\int_0^{\frac{t}{2}}(1+t-\tau)^{-\frac{3}{4}-\frac{j}{2}}
 \left\|(F_2-\hbox{div}G_2, F_4-\hbox{div}{G_4})(\tau)\right\|_{L^1}\mathrm{d}\tau\nonumber\\
&\quad+\int_{\frac{t}{2}}^t(1+t-\tau)^{-\frac{5}{4}}
 \left\||\xi|^j\left({\widehat{F_2^l}-\widehat{\hbox{div}G_2^l}}, {\widehat{F_4^l}-\widehat{\hbox{div}G_4^l}}\right)(\tau)\right\|_{L^\infty}
 \mathrm{d}\tau\nonumber\\
&:=K_0(1+t)^{-\frac{3}{4}-\frac{j}{2}}+K^j_1+K^j_2+K^j_3+K^j_4+K^j_5+K^j_6.\nonumber
\end{align}
We shall estimate $K^j_i$ with $i=1,2,\cdots, 6.$ Firstly, it
follows from H\"{o}lder inequality and \eqref{4.19} that
\begin{align}\label{4.27}
\left|K^j_1\right|&\lesssim\int_0^{\frac{t}{2}}(1+t-\tau)^{-\frac{3}{4}-\frac{j+1}{2}}
 \left\|(n^+, n^-)(\tau)\right\|_{L^2}\left\|(u^+, u^-)(\tau)\right\|_{L^2}\mathrm{d}\tau\nonumber\\
 &\lesssim
 K_0^2\int_0^{\frac{t}{2}}(1+t-\tau)^{-\frac{3}{4}-\frac{j+1}{2}}(1+\tau)^{-1}
 \mathrm{d}\tau\\
 &\lesssim K_0^2(1+t)^{-\frac{3}{4}-\frac{j}{2}}.\nonumber
\end{align}
Next, for $K^j_2$, by using Lemma \ref{1interpolation}, Lemma
\ref{es-product}, \eqref{4.1}, \eqref{4.19} and \eqref{4.20}, we
have
\begin{align}\label{4.28}
\left|K^j_2\right|&\lesssim\int_{\frac{t}{2}}^t(1+t-\tau)^{-\frac{5}{4}}
\left\||\nabla^j\left({n^+u^+}, {n^-u^-}\right)(\tau)\right\|_{L^1}\mathrm{d}\tau\nonumber\\
&\lesssim\int_{\frac{t}{2}}^t(1+t-\tau)^{-\frac{5}{4}}
 \left(\left\|\nabla^j(n^+, n^-)(\tau)\right\|_{L^2}\left\|(u^+, u^-)(\tau)\right\|_{L^2}
 +\left\|(n^+, n^-)(\tau)\right\|_{L^2}\left\|\nabla^j(u^+, u^-)(\tau)\right\|_{L^2}\right)\mathrm{d}\tau\nonumber\\
 &\lesssim
 K_0\int_0^{\frac{t}{2}}(1+t-\tau)^{-\frac{3}{4}-\frac{j+1}{2}}(1+\tau)^{-1-\frac{j}{2}}\left(1+E_j^\ell(t)\right)
 \mathrm{d}\tau\\
 &\lesssim K_0(1+t)^{-1-\frac{j}{2}}\left(1+E_j^\ell(t)\right).\nonumber
\end{align}
Similar to the proofs of \eqref{4.27} and \eqref{4.28}, for $K^j_3$
and $K^j_4$, it holds that
\begin{equation}\label{4.29}
\left|K^j_3\right|+\left|K^j_4\right|\lesssim
K_0^2(1+t)^{-\frac{3}{4}-\frac{j}{2}}+K_0(1+t)^{-1-\frac{j}{2}}\left(1+E_j^\ell(t)\right).
\end{equation}
\noindent Employing similar arguments in \eqref{4.13}, we have
\begin{equation}\label{4.30}
\left|K^j_5\right|\lesssim K_0^2(1+t)^{-\frac{3}{4}-\frac{j}{2}},
\end{equation}
where we have used \eqref{4.12} and \eqref{4.19}. The last term
$K^j_6$ is much more complicated. The main idea of our approach is
to make full use of the benefit of the low-frequency and
high-frequency decomposition. To see this, by virtue of \eqref{4.7},
\eqref{4.8}, \eqref{4.19}, \eqref{4.20}, \eqref{4.25}, Lemma
\ref{1interpolation}, Lemma \ref{es-product}, Lemma \ref{lh1} and
Lemma \ref{lh2}, we can bound the term
$\left\||\xi|^j\left({\widehat{F_2^l}-\widehat{\hbox{div}G_2^l}},
{\widehat{F_4^l}-\widehat{\hbox{div}G_4^l}}\right)(\tau)\right\|_{L^\infty}$
by
\begin{align}\label{4.31}&\left\||\xi|^j\left({\widehat{F_2^l}-\widehat{\hbox{div}G_2^l}},
{\widehat{F_4^l}-\widehat{\hbox{div}G_4^l}}\right)(\tau)\right\|_{L^\infty}
\nonumber\\
\lesssim &~\Big{\|}\nabla^{j-1}\left(u^+\cdot\nabla
u^+,u^-\cdot\nabla u^-, (\nabla n^++\nabla n^-)\nabla u^+, (\nabla
n^++\nabla n^-)\nabla u^+\right)(t)\Big{\|}_{L^1}\nonumber\\
&+\Big{\|}\nabla^{j-1}\left( \nabla(n^+,n^-)
\Big(\beta^+n^++\beta^-n^-)\right)(t)\Big{\|}_{L^1} \nonumber\\
&+\Big{\|}\nabla^{j-1}\left
(\Big(\rho^+-\bar{\rho}^++\rho^--\bar{\rho}^-\Big)\nabla n^+,
\Big(\rho^+-\bar{\rho}^++\rho^--\bar{\rho}^-\Big)\nabla
n^-\right)(t)\Big{\|}_{L^1}\nonumber\\
&+\Big{\|}\nabla^{\max\{0, j-2\}}\left
(\Big(\rho^+-\bar{\rho}^++\rho^--\bar{\rho}^-\Big)\nabla^2 u^+,
\Big(\rho^+-\bar{\rho}^++\rho^--\bar{\rho}^-\Big)\nabla^2
u^-\right)(t)\Big{\|}_{L^1}\nonumber\\
\lesssim &~\|(u^+, u^-)(t)\|_{L^2}\|\nabla^j(u^+, u^-)(t)\|_{L^2}
+\|\nabla(u^+, u^-)(t)\|_{L^2}\|\nabla^{j-1}(u^+, u^-)(t)\|_{L^2}\nonumber\\
&+\left\|\nabla(u^+, u^-)(t)\right\|_{L^2}\left\|\nabla^j(n^+,
n^-)(t)\right\|_{L^2}+\left\|\nabla(n^+, n^-)(t)\right\|_{L^2}\left\|\nabla^{j}(u^+, u^-)(t)\right\|_{L^2}\\
&+\left\|\left(\beta^+n^++\beta^-n^-\right)(t)\right\|_{L^2}\left\|\nabla^j(n^+,
n^-)(t)\right\|_{L^2}
+\left\|\nabla(n^+, n^-)(t)\right\|_{L^2}\left\|\nabla^{j-1}\left(\beta^+n^++\beta^-n^-\right)(t)\right\|_{L^2}\nonumber\\
&+\left\|\left(\rho^+-\bar{\rho}^+,
\rho^--\bar{\rho}^-\right)(t)\right\|_{L^2}\left\|\nabla^j(n^+,
n^-)(t)\right\|_{L^2}\nonumber\\
&+\left\|\nabla(n^+,
n^-)(t)\right\|_{L^2}\left\|\nabla^{j-1}\left(\rho^+-\bar{\rho}^+,
\rho^--\bar{\rho}^-\right)(t)\right\|_{L^2}\nonumber\\
&+\left\|\nabla^2(u^+, u^-)(t)\right\|_{L^2}\left\|\nabla^{\max\{0,
j-2\}}\left(\rho^+-\bar{\rho}^+,
\rho^--\bar{\rho}^-\right)(t)\right\|_{L^2}\nonumber\\
&+\left\|\left(\rho^+-\bar{\rho}^+,
\rho^--\bar{\rho}^-\right)(t)\right\|_{L^2}\left\|\nabla^j(u^+,
u^-)(t)\right\|_{L^2}\nonumber\\
\lesssim
&~K_0^2(1+t)^{-1-\frac{j}{2}}+K_0(1+t)^{-\frac{3}{2}-\frac{j}{2}}E_j^\ell(t)\nonumber,
\end{align}
which implies
\begin{equation}\label{4.32}
\left|K^j_6\right|\lesssim
K_0^2(1+t)^{-1-\frac{j}{2}}+K_0(1+t)^{-\frac{3}{2}-\frac{j}{2}}E_j^\ell(t).
\end{equation}
Substituting \eqref {4.27}--\eqref{4.30} and \eqref{4.32} into
\eqref{4.26} gives
\begin{equation}\label{4.33}
\left\|\nabla^j(u^{+,l},u^{-,l})\right\|_{L^2}\lesssim
K_0^2(1+t)^{-\frac{3}{4}-\frac{j}{2}}+K_0(1+t)^{-1-\frac{j}{2}}\left(1+E_j^\ell(t)\right).
\end{equation}
Similarly, we also have
\begin{equation}\begin{split}\label{4.34}
&\|\nabla^j(\beta^+n^{+,l}+\beta^-n^{-,l})(t)\|_{L^2}+\|\nabla^{j+1}(
n^{+,l}, n^{-,l})(t)\|_{L^2}\\
&\hspace{3.5cm}\lesssim
K_0^2(1+t)^{-\frac{3}{4}-\frac{j}{2}}+K_0(1+t)^{-1-\frac{j}{2}}\left(1+E_j^\ell(t)\right).
\end{split}\end{equation}
Putting \eqref{4.33}--\eqref{4.34} into \eqref{4.24} and then
integrating the resultant inequality over $[0, t]$, we have
\begin{equation}\begin{split}\nonumber\mathcal E_j^{\ell}(t)\leq &~\text{e}^{-{D_4}t}\mathcal E_j^{\ell}(0)
+C\int_0^t\text{e}^{-{D_4}(t-\tau)}
(1+\tau)^{-\frac{3}{2}-j}\left[K_0^2+K_0\left(1+E_j^\ell(t)\right)\right]^2\mathrm{d}\tau\\
\leq& ~CK_0^2(1+t)^{-\frac{3}{2}-j}\left(1+E_j^\ell(t)\right)^2,
 \end{split}\end{equation}
which together with the smallness of $K_0$ gives
\begin{equation}\label{4.35}E_j^{\ell}(t)\leq
 CK_0.
\end{equation}
This gives \eqref{1.19}--\eqref{1.21} directly. Finally,
\eqref{1.18} follows from \eqref{4.25} with \eqref{4.35}
immediately. \par Therefore, we complete the proof of the upper
bounds on decay rates stated in Theorem \ref{1.1}.
\end{proof}
\subsection{Proof of lower bounds on decay rates}
In this subsection, we shall prove the lower bounds on decay rates
in \eqref{1.23}--\eqref{1.26}, and thus complete the proof of
Theorem \ref{1.1}.
\begin{Theorem}\label{5mainth} Assume that the  hypotheses of Theorem \ref{1mainth}
and \eqref{1.22} are in force. Then there is a positive constant
$c_1$ independent of $t$ such that for any large enough $t$, it
holds that
\begin{equation}\begin{split}\label{4.36}&\min\left\{\|\nabla^k(\rho^+-\bar{\rho}^+,\rho^--\bar{\rho}^- )(t)\|_{L^2}\right\}\\
&\quad\quad\geq c_1(1+t)^{-\frac{3}{4}-\frac{k}{2}},
\end{split}\end{equation}
for $0\leq k\leq \ell$,
\begin{equation}\begin{split}\label{4.37}&\min\left\{
\|\nabla^k(u^+, u^-)(t)\|_{L^2}+\|\nabla^k(\beta^+n^++\beta^-n^-)(t)\|_{L^2}\right\}\\
&\quad\quad\geq c_1(1+t)^{-\frac{3}{4}-\frac{k}{2}},
\end{split}\end{equation}
for $0\leq k\leq \ell$, and
\begin{equation}\label{4.38}\min\left\{\|\nabla^k(n^+, n^-)(t)\|_{L^2}\right\}\geq
c_1
(1+t)^{-\frac{1}{4}-\frac{k}{2}},
\end{equation}
for $0\leq k\leq \ell+1$.
\end{Theorem}
\begin{proof}If $t$ is large enough, it follows from \eqref{4.5}, Proposition \ref{2.3},
\eqref{1.18}--\eqref{1.21}, Lemma \ref{lh2} that
\begin{equation}\begin{split}\nonumber&\left\|\Lambda^{-1}(u^{+}, u^{-}, \beta^+n^{+}+\beta^-n^{-})(t)\right\|_{L^2}
\\ \le ~&\left\|\Lambda^{-1}(u^{+,1}, u^{-,l}, \beta^+n^{+,l}+\beta^-n^{-,l})(t)\right\|_{L^2}
+\left\|\Lambda^{-1}(u^{+,h}, u^{-,h},
\beta^+n^{+,h}+\beta^-n^{-,h})(t)\right\|_{L^2}\\ \le
~&CK_0(1+t)^{-\frac{1}{4}}+\int_0^t(1+t-\tau)^{-\frac{1}{4}}\left\|\mathcal
(F^1,
F^3)(\tau)\right\|_{L^1}\mathrm{d}\tau+\int_0^t(1+t-\tau)^{-\frac{3}{4}}
\|(G_2,
 G_4)(\tau)\|_{L^2}\mathrm{d}\tau\\
 &+C\int_0^t(1+t-\tau)^{-\frac{1}{4}} \|(F_2-\text{div} G_2,
 F_4-\text{div} G_4)(\tau)\|_{L^2}\mathrm{d}\tau+C\left\|(u^{+}, u^{-}, \beta^+n^{+}+\beta^-n^{-})(t)\right\|_{L^2}
\\ \le ~&CK_0\left((1+t)^{-\frac{1}{4}}+\int_0^t(1+t-\tau)^{-\frac{1}{4}}(1+\tau)^{-\frac{3}{2}}\mathrm{d}
\tau+\int_0^t(1+t-\tau)^{-\frac{3}{4}}(1+\tau)^{-1}\mathrm{d}\tau\right)\\
 \le ~&CK_0(1+t)^{-\frac{1}{4}},
\end{split}\end{equation}
and
\begin{equation}\begin{split}\nonumber&\min
\left\|(u^{+}, u^{-}, \beta^+n^{+}+\beta^-n^{-})(t)\right\|_{L^2}\\
\ge &~\min\left\|(u^{+,1}, u^{-,l},
\beta^+n^{+,l}+\beta^-n^{-,l})(t)\right\|_{L^2}
\\ \ge&~C_1K_0^\vartheta(1+t)^{-\frac{3}{4}}-\int_0^t(1+t-\tau)^{-\frac{1}{4}}\left\|\mathcal
(F^1,
F^3)(\tau)\right\|_{L^1}\mathrm{d}\tau-\int_0^t(1+t-\tau)^{-\frac{3}{4}}
\|(G_2,
 G_4)(\tau)\|_{L^2}\mathrm{d}\tau\\
 &-C\int_0^t(1+t-\tau)^{-\frac{1}{4}} \|(F_2-\text{div} G_2,
 F_4-\text{div} G_4)(\tau)\|_{L^2}\mathrm{d}\tau\\
\ge&~C_1K_0^\vartheta(1+t)^{-\frac{3}{4}}-CK_0^2(1+t)^{-\frac{3}{4}}\\
\ge&~c_2(1+t)^{-\frac{3}{4}},
\end{split}\end{equation}
since $\vartheta<2$ and $K_0$ is sufficiently small. These together
with the interpolation inequality
\begin{equation}\nonumber\|f\|_{L^2}\leq C\|\Lambda^{-1}f\|_{L^2}^\frac{k}{k+1}\|\nabla^kf\|_{L^2}^\frac{1}{k+1}\end{equation}
imply \eqref{4.37} immediately. Similarly, we can prove
\eqref{4.38}. Here, we omit the details for simplicity. Finally, we
turn to prove \eqref{4.36}. To begin with, by noting \eqref{4.9} and
employing similar arguments in \eqref{4.11}, we have
\begin{align*}
&\left\|P-\bar{P}\right\|_{L^2}\\
&\gtrsim\|\Lambda^{-1}\nabla P\|_{L^2}\\
&\gtrsim \left(\|\mathcal{C}^2(1,1)\big(\bar{\rho}^-
n^++\bar{\rho}^+
n^-\big)\right\|_{L^2}\\
&\quad-\left(\|\mathcal{C}^2\big(\rho^--\bar{\rho}^-\big)\nabla
n^+\right\|_{L^{\frac{6}{5}}}+\left\|\mathcal{C}^2\big(\rho^+-\bar{\rho}^+\big)\nabla
n^-\right\|_{L^{\frac{6}{5}}}\\
&\quad-\left\|\big(\mathcal{C}^2-\mathcal{C}^2(1,1)\big)\big(\bar{\rho}^-\nabla
n^++\bar{\rho}^+\nabla
n^-\big)\right\|_{L^{\frac{6}{5}}}\\
&\gtrsim K_0(1+t)^{1\frac{3}{4}}
-\left\|\big(\rho^--\bar{\rho}^-\big)\right\|_{L^3}\left\|\nabla
n^+\right\|_{L^{2}}+\left\|\big(\rho^+-\bar{\rho}^+\big)\right\|_{L^3}\left\|\nabla
n^-\right\|_{L^{2}}\\
&\quad-\left\|\big(\mathcal{C}^2-\mathcal{C}^2(1,1)\big)\right\|_{L^3}\left\|\bar{\rho}^-
\nabla n^++\bar{\rho}^+\nabla
n^-\right\|_{L^{2}}+\left\|\bar{\rho}^-
n^++\bar{\rho}^+n^-\right\|_{L^2}\\
 &\gtrsim K_0 (1+t)^{-\frac{3}{4}}- K_0^2 (1+t)^{-\frac{5}{4}}\\
 &\gtrsim K_0 (1+t)^{-\frac{3}{4}},
\end{align*}
which leads to
\begin{equation}\label{4.39}
\|(\rho^+-\bar{\rho}^+,\rho^--\bar{\rho}^- )(t)\|_{L^2}\gtrsim K_0
(1+t)^{-\frac{3}{4}}.
  \end{equation}
Similarly, we have
\begin{equation}\label{4.40}
\|\nabla(\rho^+-\bar{\rho}^+,\rho^--\bar{\rho}^- )(t)\|_{L^2}\gtrsim
K_0 (1+t)^{-\frac{5}{4}}.
  \end{equation}
For $2\leq k\leq \ell$, by employing an interpolation technique,
\eqref{1.18} and \eqref{4.40}, we obtain
\begin{align}\begin{split}\label{4.41}
&\|\nabla^k(\rho^+-\bar{\rho}^+,\rho^--\bar{\rho}^-
)(t)\|_{L^2}\\
&\quad\geq C \|\nabla(\rho^+-\bar{\rho}^+,\rho^--\bar{\rho}^-
)(t)\|_{L^2}^k\|(\rho^+-\bar{\rho}^+,\rho^--\bar{\rho}^-
)(t)\|_{L^2}^{-(k-1)}\\
&\quad\geq c_3 (1+t)^{-\frac{3}{4}-\frac{k}{2}}.
 \end{split} \end{align}
Consequently, \eqref{4.36} follows from \eqref{4.39}--\eqref{4.41}
immediately, and thus the proof of Theorem \ref{5mainth} is
completed.\par Therefore, we have completed the proof of Theorem
\ref{1mainth}.
\end{proof}

\bigskip
\appendix

%%%%%%%%%%%%%%%%%%%%%%%%%%%%%%%%%%%%%%%%%%%%%%%
\section{Analytic tools}\label{1section_appendix}
%%%%%%%%%%%%%%%%%%%%%%%%%%%%%%%%%%%%%%%%%%%%%%%
We recall the Sobolev interpolation of the Gagliardo--Nirenberg inequality.
 \begin{Lemma}\label{1interpolation}
 Let $0\le i, j\le k$, then we have
\begin{equation}\nonumber
\norm{\nabla^i f}_{L^p}\lesssim \norm{  \nabla^jf}_{L^q}^{1-a}\norm{ \nabla^k f}_{L^r}^a
\end{equation}
where $a$ satisfies
\begin{equation}\nonumber
\frac{i}{3}-\frac{1}{p}=\left(\frac{j}{3}-\frac{1}{q}\right)(1-a)+\left(\frac{k}{3}-\frac{1}{r}\right)a.
\end{equation}

Especially, while $p=q=r=2$, we have
\begin{equation}\nonumber
 \norm{\nabla^if}_{L^2}\lesssim \norm{\nabla^jf}_{L^2}^\frac{k-i}{k-j}\norm{\nabla^kf}_{L^2}^\frac{i-j}{k-j}.
\end{equation}
\begin{proof}
This is a special case of  \cite[pp. 125, THEOREM]{Nirenberg}.
\end{proof}
\end{Lemma}
Next, to estimate the $L^p$--norm of the spatial derivatives of the
product of two functions, we shall recall the following estimate:
\begin{Lemma}\label{es-product}
For any integer $k\ge1$, we have
 \begin{equation}\nonumber
  \norm{\nabla ^k(fg)}_{L^p} \lesssim \norm{f}_{L^{p_1}}\norm{\nabla ^kg}_{L^{p_2}} +\norm{\nabla ^kf}_{L^{p_3}}\norm{g}_{L^{p_4}},
 \end{equation}
and
 \begin{equation}\nonumber
  \norm{\nabla ^k(fg)-f\nabla^kg}_{L^p} \lesssim \norm{\nabla f}_{L^{p_1}}\norm{\nabla ^{k-1}g}_{L^{p_2}} +\norm{\nabla ^kf}_{L^{p_3}}\norm{g}_{L^{p_4}},
 \end{equation}
where $p, p_1, p_{2}, p_{3}, p_{4} \in[1, \infty]$ and
$$
\frac{1}{p}=\frac{1}{p_{1}}+\frac{1}{p_{2}}=\frac{1}{p_{3}}+\frac{1}{p_{4}}.
$$
\end{Lemma}
\begin{proof}
 See \cite{Kenig}.
\end{proof}
Finally, the following two lemmas concern the estimate for the
low--frequency part and the high--frequency part of $f$.
\begin{Lemma}\label{lh1} If $f\in L^r(\mathbb R^3)$ for any $2\leq r\leq\infty$, then we have
$$\|f^l\|_{L^r}+\|f^h\|_{L^r}\lesssim \|f\|_{L^r}.$$
\end{Lemma}
\begin{proof}
For $2\leq r\leq\infty$, by virtue of Young's inequality for
convolutions, for the low frequency, it holds that
\begin{equation}\label{3.361}\|f^l\|_{L^r}\lesssim \|\mathfrak{F}^{-1}\phi\|_{L^1}\|f\|_{L^r}\lesssim \|f\|_{L^r},\nonumber\end{equation}
and hence
\begin{equation}\label{3.362}\|f^h\|_{L^r}\lesssim \|f\|_{L^r}+\|f^l\|_{L^r}\lesssim \|f\|_{L^r}.\nonumber\end{equation}\end{proof}

\begin{Lemma}\label{lh2} Let $f\in H^k(\mathbb R^3)$ for any integer $k\geq 2$. Then there exists a positive constant $C_0$ such that
\begin{equation}\label{3.363}\|\nabla^j f^h\|_{L^2}\leq C_0 \|\nabla^{j+1}f\|_{L^2},\nonumber\end{equation}
and
\begin{equation}\label{3.364}\|\nabla^{j+1} f^l\|_{L^2}\leq C_0\|\nabla^{j}f\|_{L^2},\nonumber\end{equation}
for any $0\leq j\leq k-1$.
\end{Lemma}
\begin{proof} This lemma can be shown directly by the definitions of the low--frequency and
high--frequency of $f$ and the Plancherel theorem,  and thus we omit
the details.
\end{proof}

\section*{Acknowledgments}
Yin li's research is partially supported by National Natural Science
Foundation of China $\#$11926354, Natural Science Foundation of
Guangdong Province $\#$2021A1515010292, and Innovative team project
of Guangdong Province $\#$2020KCXTD024. Huaqiao Wang's research is
partially supported by National Natural Science Foundation of China
$\#$ 11901066, Natural Science Foundation of Chongqing $\#$
cstc2019jcyj-msxmX0167 and Project $\#$ 2019CDXYST0015 and $\#$ 2020
CDJQY-A040 supported by the Fundamental Research Funds for the
Central Universities. Guochun Wu's research is partially supported
by National Natural Science Foundation of China $\#$11701193,
$\#$11671086, Natural Science Foundation of Fujian Province $\#$
2018J05005, $\#$2017J01562 and Program for Innovative Research Team
in Science and Technology in Fujian Province University Quanzhou
High-Level Talents Support Plan $\#$2017ZT012. Yinghui Zhang'
research is partially supported by Guangxi Natural Science
Foundation $\#$2019JJG110003, $\#$2019AC20214, and National Natural
Science Foundation of China $\#$11771150.

\end{document}